\documentclass[12pt]{amsart}
\usepackage{amsmath,amssymb}
\usepackage{amsfonts}
\usepackage{amsthm}
\usepackage{latexsym}
\usepackage{graphicx}
\def\p{\partial}

\def\R{\mathbb{R}}

\def\vv<#1>{\langle#1\rangle}
\def\ol{\overline}

\def\1{\mathbf{1}}

\def\Comb{{\rm Comb}}

\def\Span{{\rm Span}}
\def\Spec{{\rm Spec}}
\def\Clump{{\rm Clump}}
\def\e{\epsilon}

\def\XXint#1#2{\setbox0=\hbox{$#1{#2}{\int}$}{#2}\kern-.5\wd0 }

\def\XXint#1#2#3{{\setbox0=\hbox{$#1{#2#3}{\int}$}
     \vcenter{\hbox{$#2#3$}}\kern-.5\wd0}}

\def\vv<#1>{{\left\langle#1\right\rangle}}

\def\wh{\widehat}

\def \St{{\rm St}}
\def\Length{{\rm Length}}
\def\DB{{\rm DB}}
\newtheorem{thm}{Theorem}[section]

\newtheorem{lem}{Lemma}[section]
\newtheorem{prop}{Proposition}[section]
\newtheorem{cor}{Corollary}[section]
\theoremstyle{definition}
\newtheorem{defn}{Definition}[section]
\theoremstyle{remark}

\newtheorem{rem}{Remark}[section]

\numberwithin{equation}{section}
\def \wt{\widetilde}

\def\Br{{\rm Br}}
\begin{document}
\title{Minimal Steklov eigenvalues on combinatorial graphs}

\author{Chengjie Yu$^1$}
\address{Department of Mathematics, Shantou University, Shantou, Guangdong, 515063, China}
\email{cjyu@stu.edu.cn}
\author{Yingtao Yu}
\address{Department of Mathematics, Shantou University, Shantou, Guangdong, 515063, China}
\email{18ytyu@stu.edu.cn}
\thanks{$^1$Research partially supported by GDNSF with contract no. 2021A1515010264 and NNSF of China with contract no. 11571215.}
\renewcommand{\subjclassname}{%
  \textup{2010} Mathematics Subject Classification}
\date{}
\begin{abstract}
In this paper, we study extremal problems of Steklov eigenvalues on combinatorial graphs by extending Friedman's theory [Duke Math. J. 69 (1993), no. 3, 487--525] of nodal domains for Laplacian eigenfunctions to Steklov eigenfunctions, and solve an extremal problem for Steklov eigenvalues on combinatorial graphs that is an analogue of the extremal problem solved by Friedman [Duke Math. J. 83 (1996), no. 1, 1--18.] for Laplacian eigenvalues.  More precisely, we mainly show that the minimum of the $i^{\rm th}$ Steklov eigenvalue on a connected combinatorial graph with $n$ vertices is essentially attained by a star with each arm a minimal broom when $i\not|n$, and attained by a regular comb with each tooth a minimal broom when $i|n$.
\end{abstract}
\maketitle\markboth{C. Yu \& Y. Yu}{Extremal problems for Steklov eigenvalues}
\section{Introduction}
For a compact Riemannian manifold with boundary, the Steklov operator or Dirichlet-to-Neumann map (DtN map for short) is defined to be the map sending the Dirichlet boundary data of a harmonic function on the manifold to its Neumann boundary data. The eigenvalues of the Steklov operator are called Steklov eigenvalues of the manifold. Such notions were originated in Steklov's study of liquid sloshing more than a century ago (See \cite{St,Ku}). Steklov operators or DtN maps are also  original models for inverse problems such as detecting the inside of a body by boundary measurements. Escobar \cite{Es} also found relations between Steklov eigenvalues and the Yamabe problems on Riemannian manifolds with boundary. In recently years, Fraser and Scheon  \cite{FS11} found relations between Steklov eigenfunctions and free boundary minimal submanifolds in the Euclidean ball and made important progresses on extremal problems of Steklov eigenvalues (\cite{FS11,FS16,FS19,FS20}).

The notions of Steklov operators and Steklov eigenvalues were recently extended to discrete settings independently by Hua, Huang and Wang in \cite{HHW} and Hassannezhad and Miclo in \cite{HM}. Colbois and Girouard \cite{CG14} also considered similar extensions of Steklov eigenvalues to the discrete settings when studying Steklov eigenvalues on manifolds constructed out of graphs (See also \cite{CG18}). Although the notions are new, there have been  a number of studies on exploring the properties of Steklov eigenvalues in the discrete setting. For example, in \cite{HM,HHW,HHW2,HH,He-Hua1,Pe2,Ts}, the authors considered isoperimetric control of Steklov eigenvalues. In \cite{He-Hua2,YY}, the authors considered monotonicity of Steklov eigenvalues. In \cite{Pe1}, the author gave some interesting lower bounds. In \cite{SY2}, the authors compared the Steklov eigenvalues with the Laplacian eigenvalues. In \cite{SY1}, the authors obtained a Lichnerowicz-type estimate for Steklov eigenvalues. In \cite{SY}, the authors introduced higher degree versions of Steklove operators and Steklov eigenvalues.

In this paper, motivated by the works of Fraser and Schoen in the smooth setting and Friedman's works \cite{FR93,FR96} for Laplacian eigenvalues on combinatorial graphs, and also motivated by the works \cite{He-Hua1,He-Hua2} of He and Hua, and our previous work \cite{YY}, we consider extremal problems on combinatorial graphs. Let's first recall some preliminary notions for graphs.
\begin{defn}\label{def-weighted-graph-boundary}
\begin{enumerate}
\item A triple $(G,m,w)$ is called a weighted graph where
\begin{enumerate}

\item $G$ is a simple graph with $V(G)$ and $E(G)$ the sets of vertices and edges of $G$ respectively;
\item $m:V(G)\to \R^+$ is the vertex-measure;
\item $w:E(G)\to \R^+$ is the edge-weight.
\end{enumerate}
When $m\equiv1$ and $w\equiv 1$, we say that $G$ is equipped with the unit weight.  A subgraph $G'$ of the weighted graph $G$ is viewed as a weighted graph by restricting the vertex-measure and edge-weight of $G$ to $G'$. For convenience, the edge-weight $w$ is also viewed as a symmetric function on $V(G)\times V(G)$ by zero extension.
\item A pair $(G, B)$ is called a graph with boundary where $G$ is a simple graph and $B\subset V(G)$. The set $\Omega:=V(G)\setminus B$  is called the interior of $G$. We will also denote $B$ and $\Omega$ as $B(G)$ and $\Omega(G)$ respectively if necessary.
\item A simple graph $G$ is called a combinatorial graph with boundary (or is called a combinatorial graph for simplicity) if it is equipped with the unit weight and $B(G)=\{v\in V(G)\ |\ \deg v\leq1\}$. For a nontrivial combinatorial tree $G$,  $B(G)=L(G)$ by definition, where $L(G)$ is the collection of leaves in $G$.
\end{enumerate}
\end{defn}
On a weighted finite graph $(G,B,m,w)$ with boundary, analogous to the smooth case, one can similarly define the Steklov operator $\Lambda:\R^B\to \R^B$ (See Section \ref{sec-pre} for details). The eigenvalues of $\Lambda$ are called Steklov eigenvalues of the graph and denoted as
$$0=\sigma_1(G,B,m,w)\leq \sigma_2(G,B,m,w)\leq\cdots\leq\sigma_{|B|}(G,B,m,w).$$
For simplicity, we will also write $\sigma_i(G,B,m,w)$ as $\sigma_i$, $\sigma_i(G)$ or $\sigma_i(G,B)$ if the ignored information is clear in context. When $i>|B|$, we conventionally take $\sigma_i=+\infty$. For the trivial graph $G$ with only one vertex, the vertex is a boundary vertex by definition and hence $\sigma_1(G)=0$ and $\sigma_i(G)=+\infty$ for $i\geq 2$.

The extremal problem we considered in this paper is:
\emph{Let $G$ be a connected combinatorial graph on $n$ vertices. What is the minimum of $\sigma_i(G)$?} This problem is an analogue of the problem considered by Friedman \cite{FR96} for Laplacian eigenvalues.

Note that the notion of graphs with boundary considered in this paper is more general than that in \cite{Pe1}. The main difference is that we don't require that $E(B,B)=\emptyset$ in (2) of Definition \ref{def-weighted-graph-boundary}. In fact, this kind of more general graphs with boundary was considered in \cite{CG14,CG18}. Note that if all the edges in $E(B,B)$ are deleted, we get a graph with boundary under the more special  definition in \cite{Pe1} in general. Moreover, the process of deleting the edges in $E(B,B)$ will reduce the Steklov eigenvalues by the simple monotonicity of Steklov eigenvalues in Theorem \ref{thm-mono} below. Thus, there will be no difference in general when considering minimum problems on Steklov eigenvalues for the two kinds of graphs with boundary. The reason that we take this more general notion of graphs with boundary is twofold. One is that the notion is simpler and more general, and the other is that the path with two vertices can not be viewed as a combinatorial graph with boundary  if we take the more special definition in \cite{Pe1}.

Our answer to the extremal problem for Steklov eigenvalues is similar to that of Friedman \cite{FR96} for Laplacian eigenvalues. Follows is a summary of the main results of this paper.

\begin{thm}\label{thm-main}
Let $G$ be a connected finite combinatorial graph on  $n\geq 2$ vertices. Then,
\begin{enumerate}
\item $\sigma_2(G)\geq \Lambda(\frac{n-1}{2})$. The equality holds if and only if \begin{enumerate}
    \item $G=\DB(m,2m,m)$ when $n=4m+1$;
    \item $G=\DB(m,2m+1,m)$ when $n=4m+2$;
    \item $G=\DB(m+1,2m,m+1), \DB(m,2m+2,m)\ \mbox{or}\ \DB(m,2m+1,m+1)$ when $n=4m+3$;
    \item $G=\DB(m+1,2m+1,m+1)$ when $n=4m+4$;
    \end{enumerate}
\item for $i$ with $2< i< n$ and $i\not|\ n$, $$\sigma_i(G)\geq \Lambda(m)$$
where $m=\lfloor \frac{n}{i}\rfloor$. Moreover, when  $n=im+1$, the equality holds if and only if $G$ is a star of degree $i$ such that each arm is a minimal broom $\Br(m)$ with the center of the star as the root.
\item for $i$ with $2< i< n$ and $i|n$,
\begin{equation*}
\sigma_i(G)\geq \Lambda(m-1+\theta_i)
\end{equation*}
where $m=\frac{n}{i}$ and $\theta_i=\frac{1}{4\cos^2\frac{\pi}{2i}}$. The equality holds if and only if $G=\Comb(P_i;T)$ when $i$ is even, and $G=\Comb(P_i;T)\mbox{ or }\Comb(C_i;T)$ when $i$ is odd. Here $P_i$ and $C_i$ are the path and the cycle on $i$ vertices respectively. Moreover, $T$ is the minimal broom $\Br(m-1+\theta_i)$ with the Dirichlet boundary vertex deleted and the vertex adjacent to the Dirichlet boundary vertex as the root.
\end{enumerate}
Here,
\begin{equation*}
\Lambda(l)=\left\{\begin{array}{ll}\frac1l&0<l<1\\\frac{1}{1+\lfloor\frac{l+1}2\rfloor(l-\lfloor\frac{l+1}2\rfloor)}&l\geq 1.\end{array}\right.
\end{equation*}
For the definitions of the minimal broom $\Br(l)$, dumbbell $\DB(d_1,i,d_2)$ and regular combs $\Comb(P_i;T)$ and $\Comb(C_i;T)$, see Definition \ref{def-min-broom-l}, Definition \ref{def-dumbbell} and Definition \ref{def-comb} in  Section \ref{sec-pre} respectively.
\end{thm}
Note that when $m\geq 3$ is an odd number, there are two different brooms in $\Br(m)$ (See Lemma \ref{lem-broom-l}). Thus, when $m$ is an odd number not less than $3$, there are $i+1$ different stars of degree $i$ with all arms minimal brooms $\Br(m)$. So, in  Theorem \ref{thm-main}, when $n=im+1$ with $m$ an odd number not less than $3$, there are $i+1$ different graphs so that the equality $\sigma_i(G)=\Lambda(m)$ holds. When $n=im+s$ with $2\leq s\leq i-1$, the same as in \cite{FR96}, we are not able to completely characterize the graphs such that the equality $\sigma_i(G)=\Lambda(m)$ holds. However, the equality holds for example on the star with degree $i+s-1$ such that $i$ arms  are minimal brooms $\Br(m)$ and $s-1$ arms are paths of length $1$.

Our strategy to solve the extremal problem for Steklov eigenvalues on combinatorial finite graphs is similar to the strategy by Friedman \cite{FR96} solving the corresponding extremal problems for Laplacian eigenvalues. First by the following simple monotonicity of Steklov eigenvalues that is different with  the monotonicity in \cite{He-Hua2,YY}, we reduce the extremal problems to trees.
\begin{thm}\label{thm-mono} Let $(\wt G, \wt B,m,w)$ be a weighted finite graph and $(G,B)$ be such that $G$ is a subgraph of $\wt G$ and $B\supset \wt B$. Then
\begin{equation}\label{eq-mono}
\sigma_i(\wt G,\wt B)\geq\sigma_i(G,B)
\end{equation}
for $i=1,2,\cdots, |\wt B|$. In particular, if $\wt G$ is a finite combinatorial graph and $G$ is a spanning subgraph of $\wt G$,
Then,
\begin{equation}
\sigma_i(\wt G)\geq\sigma_i(G)
\end{equation}
for all $i\geq 0$.
\end{thm}
After reducing the extremal problem to trees, following the strategy by Friedman \cite{FR96}, we then study the minimal problem of the first Steklov eigenvalues with vanishing Dirichlet boundary data on trees (See Section \ref{sec-min-lambda} for details). In \cite{FR96}, Friedman showed that the minimal tree for the first Laplacian eigenvalues with vanishing Dirichlet boundary data is essentially a path. However, it is rather different for the Steklov case. The minimal tree for this case is essentially a minimal broom (See Theorem \ref{thm-lambda-1} for details). This makes the minimal graphs for Steklov eigenvalues become more complicated than those for Lapalcian eigenvalues in \cite{FR96}.

Next, by developing a theory of nodal domains for Steklov eigenvalues that is similar to the theory of nodal domains for Laplacian eigenvalues established by Friedman \cite{FR93}, we reduce the extremal problem for Steklov eigenvalues on trees to the minimum problem of the first Steklov eigenvalues with vanishing Dirichlet boundary data on its nodal domains which has been solved in the last step.

The organization of the rest of the paper is as follows. In Section \ref{sec-pre}, we introduce some preliminary notions and symbols on analysis of graphs; In Section \ref{sec-mono}, we prove the monotonicity in Theorem \ref{thm-mono} and discuss its rigidity; In Section \ref{sec-nodal}, we study nodal domains for Steklov eigenfunctions; In Section \ref{sec-min-lambda}, we obtain the minimum for the first Steklov eigenvalues with vanishing Dirichlet boundary data on trees which is an analogue of Lemma 3.1 in \cite{FR96} and plays an important role in later applications; In Section \ref{sec-clump}, we recall the notion of clump numbers of trees in \cite{FR96}, explore some of its properties and recall the results in \cite{FR96} about reducing clump numbers of trees by removing edges.  In Section \ref{sec-min-sigma}, we prove Theorem \ref{thm-main}. A crucial step in the proof of Theorem \ref{thm-main} is to relate the first positive Steklov eigenvalue of a tree to its clump number (See Theorem \ref{thm-Steklov-clump}) which is an analogue of Lemma 4.1 in \cite{FR96} for the first positive Laplacian eigenvalue.
\section{Preliminaries}\label{sec-pre}
In this section, we recall some notions and symbols on spectral analysis of graphs.

Let $(G,m,w)$ be a weighted finite graph. As in \cite[Definition 2.1]{Ba}, a skew symmetric function $\alpha:V(G)\times V(G)\to \R$ such that $\alpha(x,y)=0$ when $x\not\sim y$ is called a flow on the graph $G$. The collection of all flows on $G$ is denoted as $A^1(G)$. Equipped with the spaces $\R^{V(G)}$ and $A^1(G)$ the natural inner products
\begin{equation}
\vv<f,g>_{G}=\sum_{x\in V(G)}f(x)g(x)m_x
\end{equation}
and
\begin{equation}
\vv<\alpha,\beta>_{G}=\sum_{\{x,y\}\in E(G)}\alpha(x,y)\beta(x,y)w_{xy}
\end{equation}
respectively. When $A\subset V(G)$ and $S\subset E(G)$, we define
\begin{equation}
\vv<f,g>_A=\sum_{x\in A}f(x)g(x)m_x
\end{equation}
and
\begin{equation}
\vv<\alpha,\beta>_{S}=\sum_{\{x,y\}\in S}\alpha(x,y)\beta(x,y)w_{xy}.
\end{equation}
Let $d:\R^{V(G)}\to A^1(G)$ be the map:
\begin{equation}
df(x,y)=\left\{\begin{array}{ll}f(y)-f(x)&x\sim y\\
0&x\not\sim y.
\end{array}\right.
\end{equation}
Let $d^*:A^1(G)\to \R^{V(G)}$ be the adjoint operator of $d$ and define the Laplacian operator of $G$ as
\begin{equation}
\Delta_Gf=-d^*df.
\end{equation}
By direct computation, it is not hard to check that
\begin{equation}
\Delta_Gf(x)=\frac{1}{m_x}\sum_{y\in V(G)}(f(y)-f(x))w_{xy}
\end{equation}
for any $x\in V(G)$ and $f\in \R^{V(G)}$. By the definition of $\Delta_G$, it is clear that
\begin{equation}\label{eq-Green-1}
-\vv<\Delta_G f, g>_G=\vv<df,dg>_{G}
\end{equation}
for any $f,g\in \R^{V(G)}$. So $-\Delta_G$ is a nonnegative self-adjoint operator on $\R^{V(G)}$. Its eigenvalues are denoted as
$$0=\mu_1(G)\leq\mu_2(G)\leq\cdots\leq \mu_{|V(G)|}(G)$$
and are called the Laplacian eigenvalues of $G$.

Let $(G,B,m,w)$ be a weight finite graph with boundary. For any $f\in \R^{V(G)}$, define
\begin{equation}\label{eq-normal-derivative}
\frac{\p f}{\p n}(x)=\frac{1}{m_x}\sum_{y\in V(G)}(f(x)-f(y))w_{xy}=-\Delta_Gf(x)
\end{equation}
for any $x\in B$.  The operator is regarded as the discrete version of taking normal derivative on the boundary of a Riemannian manifold. By \eqref{eq-Green-1} and \eqref{eq-normal-derivative}, we have the following Green's identity:
\begin{equation}\label{eq-Green-2}
\vv<df,dg>_{G}=\vv<\frac{\p f}{\p n},g>_{B}-\vv<\Delta_G f,g>_\Omega.
\end{equation}
This may be the reason that one defines $\frac{\p f}{\p n}$ as $-\Delta_G f\big|_{B}$.

A real number $\sigma$ is called a Steklov eigenvalue of $(G,B,m,w)$ if the following boundary value problem:
\begin{equation}
\left\{\begin{array}{ll}\Delta_Gf(x)=0& x\in \Omega\\
\frac{\p f}{\p n}(x)=\sigma f(x)& x\in B
\end{array}\right.
\end{equation}
has nonzero solutions. A nonzero solution $f$ of the boundary value problem is called a Steklov eigenfunction for the eigenvalue $\sigma$. It is not hard the see that the Steklov eigenvalues are the eigenvalues of the so called Steklov operator or DtN map:
$$\Lambda:\R^{B}\to \R^B, \Lambda(u)=\frac{\p \wh u}{\p n }$$
where $\wh u$ means the harmonic extension of $u$. That is, $\wh u$ satisifies the following Dirichlet boundary value problem:
\begin{equation}
\left\{\begin{array}{ll}\Delta_G\wh u(x)=0&x\in \Omega\\
\wh u(x)=u(x)&x\in B.
\end{array}\right.
\end{equation}
By \eqref{eq-Green-2}, we know that
\begin{equation}
\vv<\Lambda(u),v>_{B}=\vv<d\wh u,d\wh v>_{G}
\end{equation}
for any $u,v\in \R^B$. So, $\Lambda$ is a nonnegative self-adjoint operator on $\R^B$, and its eigenvalues are denoted as
$$0=\sigma_1(G,B,m,w)\leq \sigma_2(G,B,m,w)\leq \cdots\leq \sigma_{|B|}(G,B,m,w).$$
For simplicity, we also denote $\sigma_i(G,B,m,w)$ as $\sigma_i,\sigma_i(G),\sigma_i(G,B)$ when the ignored information is clear in context.

To extend Friedman's theory of nodal domains for Laplacian eigenfunction to Steklov eigenfunctions, we need to consider Steklov eigenvalues with vanishing Dirichlet boundary data as in \cite{He-Hua2,YY}. So, we introduce the notion of graph with boundary and Dirichlet boundary for convenience.
\begin{defn}
A triple $(G,B, B_D)$ is called a graph with boundary $B$ and Dirichlet boundary $B_D$ if $G$ is a simple graph, $B, B_D\subset V(G)$ and $B\cap B_D=\emptyset$. $\Omega_D:=V(G)\setminus B_D$ is called the Dirichlet interior of $G$ and $\Omega=\Omega_D\setminus B$ is called the interior of $G$. An edge in $E(B_D,
\Omega_D)$ is called a Dirichlet boundary edge. We will also denote $B,B_D,\Omega$ and $\Omega_D$ as $B(G),B_D(G),\Omega(G)$ and $\Omega_D(G)$ respectively if necessary.
\end{defn}
Let $(G,B,B_D,m,w)$ be a weighted finite graph with boundary $B$ and Dirichlet boundary $B_D$. A real number $\lambda$ is called a Steklov eigenvalue of $(G,B,B_D,m,w)$ with vanishing Dirichlet boundary data on $B_D$ if the following boundary value problem:
\begin{equation}
\left\{\begin{array}{ll}\Delta_Gf(x)=0& x\in \Omega\\
\frac{\p f}{\p n}(x)=\lambda f(x)& x\in B\\
f(x)=0&x\in B_D
\end{array}\right.
\end{equation}
has nonzero solutions. A nonzero solution $f$ of the boundary value problem is called a Steklov eigenfunction with vanishing Dirichlet boundary data on $B_D$ for the eigenvalue $\lambda$. The Steklov eigenvalues with vanishing Dirichlet boundary data on $B_D$ are the eigenvalues of the following Steklov operator with vanishing Dirichlet boundary data on $B_D$:
$$\Lambda_0:\R^B\to\R^B,\ \Lambda_0(u)=\frac{\p \zeta_u}{\p n}$$
where $\zeta_u$ satisfies
\begin{equation}
\left\{\begin{array}{ll}\Delta_G\zeta_u(x)=0&x\in \Omega\\
\zeta_u(x)=u(x)&x\in B\\
\zeta_u(x)=0&x\in B_D.
\end{array}\right.
\end{equation}
By \eqref{eq-Green-2},
\begin{equation}
\vv<\Lambda_0(u),v>_{B}=\vv<d\zeta_u,d\zeta_v>_G
\end{equation}
for any $u,v\in \R^B$. So, $\Lambda_0:\R^B\to\R^B$ is a nonnegative self-adjoint operator on $\R^B$, and its eigenvalues are denoted as:
$$0\leq \lambda_1(G,B,B_D,m,w)\leq \lambda_2(G,B,B_D,m,w)\leq\cdots\leq\lambda_{|B|}(G,B,B_D,m,w).$$
For simplicity, we also denote $\lambda_i(G,B,B_D,m,w)$ as $\lambda_i,\lambda_i(G),\lambda_i(G,B,B_D)$ when the ignored information is clear in context. By the Dirichlet principle, it is clear that
\begin{equation*}
\lambda_1(G,B,B_D)=\min\left\{\frac{\vv<df,df>_{G}}{\vv<f,f>_B}\ \bigg|\ f\in \R^{V(G)}\ \mbox{with }f|_{B}\not\equiv 0\ \mbox{and }f|_{B_D}=0.\right\}.
\end{equation*}
Thus, if $G$ is connected and $B_D$ is nonempty, then $\lambda_1(G,B,B_D)>0$. It is also clear that the measures of the vertices in $B_D$ do not affect the map $\Lambda_0$ and the eigenvalues $\lambda_i(G,B,B_D)$.

Next, we introduce the notion of a broom.
\begin{defn}\label{def-broom}
Let $P$ be the path of length $i$. The tree formed by adding a Dirichlet boundary edge to one of the end points of $P$ with length $l$ (i.e. with weight $\frac1l$) and adding $d$ boundary vertices to the other end point of $P$ is denoted as $\Br(l,i,d)$ which is called a broom with parameters $(l,i,d)$. When $d=0$, the other end point of $P$ is viewed as the boundary vertex. The Dirichlet boundary vertex is viewed as the root of the broom. All the vertices and edges of $\Br(l,i,d)$ except the Dirichlet boundary edge are considered to be of unit weight.
\end{defn}
Note that $\Br(l,i,1)=\Br(l,i+1,0)$ by definition. The Steklov eigenvalues with vanishing Dirichlet boundary data of a broom can be computed directly.
\begin{lem}\label{lem-lambda1-broom}
For any real number $l>0$ and nonnegative integers $i$ and $d$,
\begin{equation*}
\lambda_1(\Br(l,i,d))=\left\{\begin{array}{ll}\frac{1}{l+i}&d=0\\
\frac{1}{1+(l+i)d}&d\geq 1.
\end{array}\right.
\end{equation*}
Moreover, let $v_0\sim v_1\sim\cdots\sim v_i$ be the path $P$ and $o\sim v_0$ in $\Br(l,i,d)$, and $u_1,u_2,\cdots,u_d$ be the boundary vertices of $\Br(l,i,d)$ when $d>0$, where $o$ is the root of $\Br(l,i,d)$. Then, the eigenfunction $f$ for $\lambda_1$ is given by
\begin{equation}
f(x)=\left\{\begin{array}{ll}0&x=o\\
\frac{l+j}{l+i}&x=v_j \mbox{ with }j=0,1,\cdots,i
\end{array}\right.
\end{equation}
when $d=0$, and
\begin{equation}
f(x)=\left\{\begin{array}{ll}0&x=o\\
\frac{(l+j)d}{1+(l+i)d}&x=v_j \mbox{ with }j=0,1,\cdots,i\\
1&x=u_j\mbox{ with }j=1,2,\cdots,d\\
\end{array}\right.
\end{equation}
when $d>0$. Furthermore, when $d\geq 2$,
$$\lambda_2(\Br(l,i,d))=\cdots=\lambda_d(\Br(l,i,d))=1$$
with the eigenspace generated by
\begin{equation*}
f_j(x)=\left\{\begin{array}{ll}1&x=u_1\\
-1&x=u_j\\
0&\mbox{otherwise}
\end{array}\right.
\end{equation*}
for $j=2,3,\cdots, d$.
\end{lem}
\begin{proof}The conclusion can be shown by direct verification.
\end{proof}
Next, we introduce the notion of minimal brooms. It plays an important role in Section \ref{sec-min-lambda}.
\begin{defn}\label{def-min-broom}
For any nonnegative integer $n$ and  positive real number $l$, we denote $\Br(l,n)$ the brooms that have the minimal $\lambda_1$ among all brooms $\Br(l,i,d)$ with $i+d=n$. We call $\Br(l,n)$ the minimal brooms with parameters $(l,n)$ and denote $\lambda_1(\Br(l,n))$ as $\Lambda(l,n)$.
\end{defn}
By Lemma \ref{lem-lambda1-broom}, we have the following conclusion.
\begin{lem}\label{lem-broom-l-n}
The function $\Lambda(l,n)$ is strictly decreasing on $l$ and $n$. More precisely, when $n=0,1$, $\Br(l,n)=\Br(l,n,0)$, and when $n\geq 2,$
\begin{equation*}
\begin{split}
\Br(l,n)=\left\{\begin{array}{ll}\Br(l,0,n)&l\geq n\\
\Br\left(l,\lfloor\frac{n-l}{2}\rfloor,n-\lfloor\frac{n-l}{2}\rfloor\right)&l<n\mbox{ and }\{\frac{n-l}{2}\}<\frac12\\
\Br\left(l,\lceil\frac{n-l}{2}\rceil,n-\lceil\frac{n-l}{2}\rceil\right)&l<n\mbox{ and }\{\frac{n-l}{2}\}>\frac12\\
\Br\left(l,\lfloor\frac{n-l}{2}\rfloor,n-\lfloor\frac{n-l}{2}\rfloor\right)\mbox{ or }&\\
\Br\left(l,\lceil\frac{n-l}{2}\rceil,n-\lceil\frac{n-l}{2}\rceil\right)&l<n\mbox{ and }\{\frac{n-l}{2}\}=\frac12.
\end{array}\right.
\end{split}
\end{equation*}
Moreover
\begin{equation*}
\Lambda(l,n)=\left\{\begin{array}{ll}\frac{1}{l+n}&n=0,1\\
\frac{1}{1+nl}&l\geq n\geq 2\\
\frac{1}{1+(l+\lfloor\frac{n-l}{2}\rfloor)(n-\lfloor\frac{n-l}{2}\rfloor)}&n\geq 2,l<n \mbox{ and } \{\frac{n-l}{2}\}\leq \frac12\\
\frac{1}{1+(l+\lceil\frac{n-l}{2}\rceil)(n-\lceil\frac{n-l}{2}\rceil)}&n\geq 2,l<n\mbox{ and } \{\frac{n-l}{2}\}> \frac12.
\end{array}\right.
\end{equation*}
Here $\{x\}:=x-\lfloor x\rfloor$ for $x\in \R$.
\end{lem}
We now introduce the notion of minimal brooms with total length $l$.
\begin{defn}\label{def-min-broom-l}
For a real number $l>0$, define $$\Br(l)=\left\{\begin{array}{ll}\Br(1,l-1)& \mbox{$l$ is an integer}\\
\Br(\{l\},\lfloor l\rfloor)&\mbox{Otherwise}
\end{array}\right.$$ and $$\Lambda(l)=\lambda_1(\Br(l)).$$
\end{defn}
By Lemma \ref{lem-broom-l-n}, we have the following conclusions for minimal brooms with total length $l$.
\begin{lem}\label{lem-broom-l}
The function $\Lambda(l)$ is strictly decreasing. More precisely,
\begin{equation*}
\Br(l)=\left\{\begin{array}{ll}\Br(l-\lceil l-1\rceil,\lceil l-1\rceil,0)&l\leq 2;\\
\Br(1,m-1,m)&l>2\mbox{ and $l=2m$ with }m\in \mathbb N \\
\Br(1,m-1,m+1)\mbox{ or}&\\
\Br(1,m,m)&l>2\mbox{ and $l=2m+1$ with }m\in \mathbb N\\
\Br(\alpha,m,m)&l=2m+\alpha\mbox{ with }m\in \mathbb N,\alpha\in (0,1)\\
\Br(\alpha,m,m+1)&l=2m+1+\alpha\mbox{ with }m\in \mathbb N,\alpha\in (0,1)\\
\end{array}
\right.
\end{equation*}
and
\begin{equation*}
\Lambda(l)=\left\{\begin{array}{ll}\frac1l&l\leq 2\\
\frac1{1+m^2}&l>2\mbox{ and $l=2m$ with }m\in \mathbb N \\
\frac{1}{1+m(m+1)}&l>2\mbox{ and $l=2m+1$ with }m\in \mathbb N\\
\frac{1}{1+m(m+\alpha)}&l=2m+\alpha\mbox{ with }m\in \mathbb N,\alpha\in (0,1)\\
\frac{1}{1+(m+\alpha)(m+1)}&l=2m+1+\alpha\mbox{ with }m\in \mathbb N,\alpha\in (0,1).\\
\end{array}
\right.
\end{equation*}
In particular, when $l\geq 1$,
\begin{equation*}
\Lambda(l)=\frac{1}{1+\lfloor\frac{l+1}2\rfloor(l-\lfloor\frac{l+1}2\rfloor)},
\end{equation*}
and when $l\in \mathbb N$,
$$\Lambda(l)=\frac{1}{1+\lfloor\frac{l^2}{4}\rfloor}.$$
\end{lem}
Heuristically, the wedge-sum of two brooms on their roots is called a dumbbell. For the definition of wedge-sums of graphs, see \cite{He-Hua1,YY}.
\begin{defn}\label{def-dumbbell}
Let $P$ be a path of length $i$. The combinatorial tree formed by adding $d_0$ edges and $d_1$ edges adjacent to the two end vertices of $P$ respectively is called a dumbbell with parameters $(d_0,i,d_1)$ and is denoted as $\DB(d_0,i,d_1)$.
\end{defn}
The Steklov spectrum of the a dumbbell can be computed directly.
\begin{lem}\label{lem-dumbbell}
Let $d_0,i,d_1$ be positive integers. Then, the Steklov eigenvalues of $\DB(d_0,i,d_1)$ are given by
$$\sigma_2=\frac{d_0+d_1}{d_0+d_1+id_0d_1},\ \sigma_3=\sigma_4=\cdots=\sigma_{d_0+d_1}=1.$$
Moreover, let $x_0\sim x_1\sim \cdots\sim x_i$ be the path $P$ in $\DB(d_0,i,d_1)$, $u_1,u_2,\cdots,u_{d_0}$ be the boundary vertices joining to $x_0$ and $v_1,v_2,\cdots,v_{d_1}$ be the boundary vertices joining to $x_i$. Then, the eigenfunction of $\sigma_2$ is given by
\begin{equation*}
f(x)=\left\{\begin{array}{ll}-d_1&x=u_j,j=1,2,\cdots d_0\\
d_0&x=v_j,j=1,2,\cdots,d_1\\
\frac{d_0d_1(jd_0-(i-j)d_1)}{d_0+d_1+d_0d_1i}& x=x_j,j=0,1,2\cdots,i,
\end{array}\right.
\end{equation*}
and the eigenspace for the eigenvalue $1$ is generated by the following eigenfunctions:
\begin{equation*}
g_j(x)=\left\{\begin{array}{ll}1&x=u_{d_0}\\
-1&x=u_j\\
0& \mbox{otherwise}
\end{array}\right.
\end{equation*}
for $j=1,2,\cdots,d_0-1$, and
\begin{equation*}
h_k(x)=\left\{\begin{array}{ll}1&x=v_{d_1}\\
-1&x=v_k\\
0& \mbox{otherwise}
\end{array}\right.
\end{equation*}
for $k=1,2,\cdots,d_1-1$.
\end{lem}
\begin{proof}
The conclusion can be obtained by direct verification.
\end{proof}
Next we introduce the notions of stars and combs which extend the corresponding notions in \cite{FR96}.
\begin{defn}\label{def-star}
 Let $T_1,T_2,\cdots,T_r$ be $r$ rooted trees with $o_1,o_2,\cdots,o_r$ their roots respectively. Then, the wedge-sum of $T_1,T_2,\cdots,T_r$ on their roots is called a star of degree $i$ with arms $T_1,T_2,\cdots,T_r$, and is denoted as $\St(T_1,T_2,\cdots,T_r)$.  The identified vertex of the roots is called the center of the star.  When $T_i$ is a path of length $l_i$ with the root $o_i$ one of its end vertices for $i=1,2,\cdots,r$, $\St(T_1,T_2,\cdots,T_r)$ is simply denoted as $\St(l_1,l_2,\cdots,l_r)$. If $l_1=l_2=\cdots=l_r=l$, the star $\St(l_1,l_2,\cdots,l_r)$ is simply denoted as $\St(r;l)$.
\end{defn}
The Steklov spectrum of a star of degree $i$ with each arms a minimal broom $\Br(m)$ can be computed directly.
\begin{lem}\label{lem-star-broom}
For any integers $i\geq 2$ and $m\geq 1$, let $T_j$ be a minimal broom $\Br(m)$ for $j=1,2,\cdots, i$ and the combinatorial tree $G=\St(T_1,T_2,\cdots,T_i)$ . Then
$$\sigma_2(G)=\sigma_3(G)=\cdots=\sigma_i(G)=\Lambda(m)$$
and the other positive Steklov eigenvalues of $G$ are $1$. Moreover, the eigenspace for the eigenvalue $\Lambda(m)$ is generated by the eigenfunctions
\begin{equation}
g_j(x)=\left\{\begin{array}{ll}f(x)&x\in V(T_i)\\
f_j(x)&x\in V(T_j)\\
0&\mbox{otherwise}
\end{array}\right.
\end{equation}
for $j=1,2\cdots i-1$, where $f$ is a fixed first Steklov eigenfunction of $T_i$ with vanishing Dirichlet boundary data and $f_j$ is the first Steklov eigenfunction of $T_j$ with vanishing Dirichlet boundary data such that
$$\vv<f,1>_{B(T_i)}+\vv<f_j,1>_{B(T_j)}=0$$
for $j=1,2,\cdots,i-1$. Furthermore, let $B(T_j)=\{u_{j1},u_{j2},\cdots,u_{jd_j}\}$ for $j=1,2,\cdots, i$. Then, the eigenspace of the eigenvalue $1$ is generated by the eigenfunctions
\begin{equation*}
h_{jk}(x)=\left\{\begin{array}{ll}1&x=u_{jd_j}\\
-1&x=u_{jk}\\
0&\mbox{otherwise}
\end{array}\right.
\end{equation*}
for $j=1,2,\cdots,i$ and $k=1,2,\cdots,d_j-1$.
\end{lem}
\begin{proof}
The conclusion can be obtained by direct verification using Lemma \ref{lem-lambda1-broom}.
\end{proof}
\begin{defn}\label{def-comb}
Let $G$ be a connected subgraph of the connected graph $\wt G$. For each $x\in V(G)$, we denote by $\wt G_x$ the connected component of $\wt G$ with all the edges in $G$ deleted that contains $x$.
\begin{enumerate}
\item If for any $x, y\in V(G)$ with $x\neq y$, $V(\wt G_x)\cap V(\wt G_y)=\emptyset$, we call $\wt G$ a comb over $G$. The subgraph $G$ is called the base of the comb and $\wt G_x$ is called the tooth of the comb at the base vertex $x\in V(G)$;
\item Let $\wt G$ be a  comb over $G$ and $T$ be a rooted tree with $o$ as its root. If for each $x\in V(G)$, the tooth $(\wt G_x,x)$ is isomorphic to $(T,o)$ as rooted trees, then $\wt G$ is called a regular comb over $G$ with tooth $T$ and is denoted as $\Comb(G;T)$.
\end{enumerate}
\end{defn}
Similar to Lemma 5.3 in \cite{YY}, the Steklov eigenvalues and eigenfunctions of a regular comb can be computed directly as follows.
\begin{lem}\label{lem-comb}
Let the combinatorial graph $\wt G=\Comb(G;T)$ be a regular comb over the nontrivial connected graph $G$ with tooth the nontrivial rooted tree $(T,o)$. Let $$0=\mu_1<\mu_2\leq\cdots\leq\mu_{|V(G)|}$$ be the Laplacian eigenvalues of the connected graph $G$, and $T_1=T$ and $T_i$ be the tree formed by adding a Dirichlet boundary edge $\{o',o\}$ of weight $\mu_i$ to $T$ adjacent to $o$ for $i=2,3,\cdots,|V(G)|$.  Let $f$ be an eigenfunction of $(T_1,L(T_1)\setminus \{o\})$ when $i=1$ or $(T_i,L(T_i)\setminus\{o'\},\{o'\})$ when $i\geq 2$, and $g$ be an eigenfunction of $\mu_i(G)$. Moreover, let
\begin{equation}
h(v)=g(x)f(\varphi_x(v))
\end{equation}
when $v\in \wt G_x$ with $x\in V(G)$. Here $\varphi_x:(\wt G_x,x)\to (T,o)$ is the isomorphism of the rooted trees. Then, $h$ is a Steklov eigenfunction for $\wt G$ with the same eigenvalue as $f$.
\end{lem}
\begin{proof} The conclusion is clearly true for $i=1$ by noting that $g\equiv c$ in this case. When $i\geq 2$, for any $x\in V(G)$,
\begin{equation}
\begin{split}
(\Delta_{\wt G}) h(x)=&(\Delta_Gh)(x)+(\Delta_{\wt G_x}h)(x)\\
=&f(o)(\Delta_Gg)(x)+g(x)(\Delta_{T}f)(o)\\
=&-\mu_if(o)g(x)+g(x)[(\Delta_{T_i}f)(o)-(f(o')-f(o))\mu_i]\\
=&-\mu_if(o)g(x)+\mu_if(o)g(x)\\
=&0
\end{split}
\end{equation}
by noting that $\Delta_{T_i}f(o)=0$ and $f(o')=0$.

Moreover, when $v\in \wt G_x$ and $v\neq x$,
\begin{equation}
(\Delta_{\wt G}h)(v)=(\Delta_{\wt G_x}h)(v)=g(x)(\Delta_Tf)(\varphi_x(v))=g(x)(\Delta_{T_i}f)(\varphi_x(v)).
\end{equation}
Thus, we get the conclusion of the lemma.
\end{proof}
\begin{lem}\label{lem-sigma-lambda}
Let $(G,B,m,w)$ be a weighted connected finite graph with boundary and $z\in \Omega$.  Let $\wt G$ be a graph formed by adding a Dirichlet boundary edge $\{o,z\}$ of weight $w_{oz}$ to $G$ adjacent to $z$. Then,
\begin{equation}
\sigma_2(G)>\lambda_1(\wt G).
\end{equation}
\end{lem}
\begin{proof}
Let $f$ be an eigenfunction of $\sigma_2(G)$. Then, $\vv<f,1>_B=0$ and
\begin{equation}
\sigma_2(G)=\frac{\vv<df,df>_{G}}{\vv<f,f>_B}.
\end{equation}
Let $g\in \R^{V(\wt G)}$ be such that $g|_{V(G)}=f-f(z)$ and $g(o)=0$. Then,
\begin{equation}
\lambda_1(\wt G)\leq\frac{\vv<dg,dg>_{\wt G}}{\vv<g,g>_{B}}=\frac{\vv<df,df>_{G}}{\vv<f,f>_B+f^2(z)\sum_{x\in B}m_x}\leq \sigma_2(G).
\end{equation}
When $\lambda_1(\wt G)=\sigma_2(G)$, we  know that $f(z)=0$ and $g$ is an eigenfunction of $\wt G$ for $\lambda_1(\wt G)$. However, this is impossible by Theorem \ref{thm-first-eigenfunction}.
\end{proof}
Combining Lemma \ref{lem-comb} and Lemma \ref{lem-sigma-lambda}, we can get  $\sigma_{|V(G)|}(\wt G)$ for $\wt G=\Comb(G;T)$.
\begin{cor}\label{cor-sigma-i-comb}
Let $\wt G=\Comb(G;T)$ and the notations be the same as in Lemma \ref{lem-comb}.  Then, $\sigma_{|V(G)|}(\wt G)=\lambda_1\left(T_{|V(G)|}\right)$.
\end{cor}
\begin{proof} By Lemma \ref{lem-sigma-lambda},
$$\sigma_2(T_1,L(T_1)\setminus\{o\})>\lambda_1(T_i).$$
for $i=2,3,\cdots,|V(G)|$. So, by Lemma \ref{lem-comb}, we know that $\sigma_1(\wt G)=0$ and $$\sigma_i(\wt G)=\lambda_1(T_i)$$ for $i=2,3,\cdots, |V(G)|$.
\end{proof}
The following result was essentially obtained in \cite{FR96} and will be used when handling the case $i|n$ in Theorem \ref{thm-main}. Here, we give a more elementary proof of the result for completeness.
\begin{prop}\label{prop-min-top}
Let $G$ be a connected graph on $n\geq 2$ vertices and $P_n$ be a path on $n$ vertices that are both equipped with the unit weight. Then,
\begin{equation}
\mu_n(G)\geq\mu_n(P_n)=4\cos^2\frac{\pi}{2n}.
\end{equation}
The equality holds if and only $G$ is a path when $n$ is even, and is either a cycle or a path when $n$ is odd.
\end{prop}
\begin{proof}
If there is a vertex $x\in V(G)$ such that $\deg(x)\geq 3$, let $y_1,y_2,y_3\in V(G)$ be such that $x\sim y_i$ for $i=1,2,3$. Let $f\in \R^{V(G)}$ be such that
\begin{equation}
f(v)=\left\{\begin{array}{ll}-1&v=y_1,y_2,y_3\\
3&v=x\\
0& \mbox{otherwise.}
\end{array}\right.
\end{equation}
Then,
 $$\mu_n(G)\geq \frac{\vv<df,df>_{G}}{\vv<f,f>_G}\geq \frac{\sum_{i=1}^3(f(y_i)-f(x))^2}{f^2(x)+\sum_{i=1}^3f^2(y_i)}=4>\mu_n(P_n).$$

If for any $x\in V(G)$, $\deg(x)\leq 2$, then $G$ is either a cycle $C_n$ or a path $P_n$. Note that $$\mu_n(C_n)=4>\mu_n(P_n)$$ when $n$ is even and $$\mu_n(C_n)=\mu_n(P_n)=4\cos^2\frac{\pi}{2n}$$
when $n$ is odd (See \cite[P. 9]{BR}). So, we complete the proof of the theorem.
\end{proof}
Finally, recall the notion of geometric representation of a graph introduced by Friedman \cite{FR93}.
\begin{defn}
\begin{enumerate}
  \item For a simple graph $G$, let $K(G)$ be the one dimensional simplicial complex with the vertex set $V(G)$ corresponding the set of 0-simplexes in $K(G)$ and the edge set $E(G)$ corresponding the set of 1-simplexes such that the boundary points of the 1-simplex $\{x,y\}$ are $x$ and $y$. Then, $K(G)$ is called the one dimensional simplicial complex representing $G$.
  \item A weight on an abstract one dimensional simplicial complex is to assign each 0-simplex a measure and each 1-simplex a length.
  \item Let $(G,m,w)$ be a weighted graph. Assign to each 0-simplex $x$ of $K(G)$ the measure $m_x$ and assign to each 1-simplex $\{x,y\}$ of $K(G)$ the length $l_{xy}=\frac{1}{w_{xy}}$. We will simply denote such a weighted one dimensional simplicial complex as $(K(G),m,\frac1{w})$ and call it the geometric representation of $(G,m,w)$. We also simply denote $(K(G),m,\frac1{w})$ as $K(G)$.
\item Let $(G,m,w)$ be a weighted graph and $K(G)$ be its geometric representation. We then identify each 1-simplex $\{x,y\}$ in $K(G)$ with the interval $[0,\frac{1}{w_{xy}}]$. Let $f\in \R^{V(G)}$, we denote $\wt f:|K(G)|\to \R$ the edgewisely linear extension of $f$. Here $|K(G)|$ is the underlying topological space of $K(G)$.
 \end{enumerate}
\end{defn}
Note that the geometric representation $K(G)$ of a weighted graph $(G,m,w)$ is a one dimensional Riemannian polyhedron in \cite[P. 20]{Pe} equipped the vertex-measure $m$ such that the length of the 1-simplex $\{x,y\}$ is $\frac{1}{w_{xy}}$. So $|K(G)|$ can be naturally viewed as a geodesic space. For any two points $x,y\in |K(G)|$, we denote $|xy|$ the geodesic distance of $x$ and $y$. Moreover, for any $x,y\in |K(G)|$, if $x$ and $y$ are contained in the same edge of $K(G)$, we denote by $[xy]$ the part of that edge lying between $x$ and $y$ (including $x$ and $y$) and $|[xy]|$ means the length of $[xy]$.

Conversely, from a connected open set in $|K(G)|$, we can define its induced graph.
\begin{defn}
Let $(G,B,m,w)$ be a weighted finite graph with boundary $B$ and $K(G)$ be the geometric representation of $(G,m,w)$. For any connected open subset $U$ of  $|K(G)|$, we define its induced graph $G_U$ as the weighted graph with boundary and Dirichlet boundary as follows:\\
(1) $V(G_{U})=\{x\in \ol U\ |\ x\in V(G)\ \mbox{or}\ x\in \p U\}$;\\
(2) $E(G_U)=\{\{x,y\}\ |\  x\neq y\in V(G_U)\ \mbox{lie on the same edge of }|K(G)|\ \mbox{and}\ (xy)\subset  U.\},$ where $(xy)=[xy]\setminus\{x,y\}$;\\
(3) $B_D(G_U)=\p U$;\\
(4) $B(G_U)=B(G)\cap U$;\\
(5) for any $\{x,y\}\in E(G_U)$, $w_{xy}=\frac{1}{|[xy]|}$;\\
(6) for any $x\in V(G_U)$, $m_x=\left\{\begin{array}{ll}m_x& x\not\in B_D(G_U)\\
1&x\in B_D(G_U).
\end{array}\right.$
\end{defn}

Similarly as in \cite{FR96}, we need the following two results when considering  the case $i|n$ in Theorem \ref{thm-main}. They are  essentially contained in \cite{FR96}. Here, we state them in more general settings and give their detailed proofs for completeness.
\begin{prop}\label{prop-top-eigen}
Let $(G,m,w)$ be a weighted connected bipartite finite graph. Then, the greatest Laplacian eigenvalue $\mu_{|V(G)|}(G)$ is of multiplicity one and its eigenfunction $f$ must have alternating signs which means that $f(x)f(y)<0$ when $x\sim y$.
\end{prop}
\begin{proof}
Let $\varphi$ be a function on $V(G)$ taking values in $\{1,-1\}$ which has alternating signs. The existence of $\varphi$ is guaranteed by that $G$ is bipartite.  Let $f$ be an eigenfunction for $\mu_{|V(G)|}$. Note that
\begin{equation}\label{eq-top-eigen}
\begin{split}
\mu_{|V(G)|}\geq&\frac{\vv<d(\varphi|f|),d(\varphi|f|)>_G}{\vv<\varphi|f|,\varphi|f|>_{G}}\\
=&\frac{\sum_{\{x,y\}\in E(G)}(|f(x)|+|f(y)|)^2w_{xy}}{\vv<f,f>_G}\\
\geq& \frac{\vv<df,df>_G}{\vv<f,f>_G}\\
=&\mu_{|V(G)|}.
\end{split}
\end{equation}
So, $\varphi|f|$ is an eigenfunction of $\mu_{|V(G)|}$.

Next, we want to show that $|f|>0$. Otherwise, there is a vertex $v$ with $f(v)=0$ and some $x\sim v$ with $|f(x)|>0$ since $G$ is connected. Then,
\begin{equation}
\begin{split}
0=&-\mu_{|V(G)|}\varphi(v) |f|(v)\\
=&\Delta_G(\varphi|f|)(v)\\
=&\frac{1}{m_x}\sum_{x\sim v}\varphi(x)|f(x)|w_{xv}\\
=&-\frac{\varphi(v)}{m_x}\sum_{x\sim v}|f|(x)w_{xv}\\
\neq& 0
\end{split}
\end{equation}
which is a contradiction. So $|f|>0$. Moreover, by that the inequalities in \eqref{eq-top-eigen} are equalities,
\begin{equation}
|f(x)|+|f(y)|=|f(x)-f(y)|
\end{equation}
when $x\sim y$. So $f(x)f(y)<0$ when $x\sim y$.  This shows that $f$ has alternating signs.

Suppose the multiplicity of $\mu_{|V(G)|}$ is greater than one. Then, there are two eigenfunctions $f,g$ of $\mu_{|V(G)|}$ such that $\vv<f,g>_{G}=0$. However, this is impossible because both $f$ and $g$ have alternating signs. This completes the proof of the theorem.
\end{proof}
\begin{cor}\label{cor-top-eigen}
Let $(G,m,w)$ be a weighted connected bipartite finite graph and $f$ be a Laplacian eigenfunction for $\mu_{|V(G)|}(G)$. Let $K(G)$ be the geometric representation of $G$ and for any edge $\ol{xy}$ in $|K(G)|$, let $z_{xy}\in [xy]$ be the zero point of $\wt f$. Then, for any vertex $x$,
\begin{equation}
\frac{1}{m_x}\sum_{y\sim x}\frac{1}{\left|[xz_{xy}]\right|}=\mu_{|V(G)|}(G).
\end{equation}
\end{cor}
\begin{proof}
Note that $f(x)\neq 0$ and $w_{xy}=\frac{1}{|[xy]|}$ for any edge $\{x,y\}$. So
\begin{equation}
\begin{split}
\mu_{|V(G)|}(G)f(x)=&-(\Delta_G f)(x)\\
=&\frac{1}{m_x}\sum_{y\sim x}\frac{f(x)-f(y)}{|[xy]|}\\
=&\frac{1}{m_x}\sum_{y\sim x}\frac{\wt f(x)-\wt f(z_{xy})}{|[xz_{xy}]|}\\
=&\frac{1}{m_x}\sum_{y\sim x}\frac{1}{|[xz_{xy}]|}f(x).
\end{split}
\end{equation}
Hence, $\frac{1}{m_x}\sum_{y\sim x}\frac{1}{|[xz_{xy}]|}=\mu_{|V|}(G)$.
\end{proof}
\section{A monotonicity of Steklov eigenvalues}\label{sec-mono}
In this section, we prove Theorem \ref{thm-mono} and discuss the rigidity when the equalities of \eqref{eq-mono} hold.
\begin{proof}[Proof of Theorem \ref{thm-mono}]
Let $\varphi_1=1,\varphi_2,\cdots,\varphi_{|\wt B|}\in \R^{V(\wt G)}$ be an orthogonal system of eigenfunctions of the DtN map for $(\wt G,\wt B)$ such that $\varphi_i$ is an eigenfunction of $\sigma_i(\wt G)$ for $i=1,2,\cdots |\wt B|$.  Let $\psi_1=1,\psi_2,\cdots,\psi_{|B|}\in \R^{V(G)}$ be an orthogonal system of eigenfunctions of the DtN map for $(G,B)$ such that $\psi_i$ is an eigenfunction of $\sigma_i(G)$ for $i=1,2,\cdots,|B|$. For each $2\leq i\leq |\wt B|$,
let $$\varphi=c_1\varphi_1+c_2\varphi_2+\cdots+c_{i}\varphi_i$$ with $c_1,c_2,\cdots,c_i$ not all zero such that
\begin{equation}\label{eq-orthogonal}
\vv<\varphi,\psi_j>_{B}=0
\end{equation}
for $j=1,2,\cdots,i-1$. The existence of such $c_1,c_2,\cdots,c_i$ comes from the fact that \eqref{eq-orthogonal} with $j=1,2\cdots,i-1$ form a linear homogeneous system with $i-1$ equations on $i$ unknowns: $c_1,c_2,\cdots,c_{i}$ which certainly has nonzero solutions. Then, by that $E(G)\subset E(\wt G)$ and $\wt B\subset B$,
\begin{equation}
\begin{split}
\sigma_i(\wt G,\wt B)\geq\frac{\vv<d\varphi,d\varphi>_{\wt G}}{\vv<\varphi,\varphi>_{\wt B}}
\geq\frac{\vv<d\varphi,d\varphi>_{G}}{\vv<\varphi,\varphi>_{B}}
\geq\sigma_i(G,B).
\end{split}
\end{equation}
This completes the proof of the first part of the theorem.

When $\wt G$ is a combinatorial finite graph and $G$ is its spanning subgraph, it is clear that $B(G)\supset B(\wt G)$. So, the conclusion follows. This completes the proof of the theorem.
\end{proof}
For completeness, we also discuss the rigidity for \eqref{eq-mono} to hold for all $i=1,2\cdots,|\wt B|$.
\begin{thm}\label{thm-rigidity} Let $(\wt G, \wt B,m,w)$ be a weighted connected finite graph with $|\wt B|\geq 2$ and $(G,B)$ be such that $G$ is a connected subgraph of $\wt G$ and $B\supset \wt B$. Then, the equalities of \eqref{eq-mono} hold for all $i=1,2,\cdots, |\wt B|$ if and only if all the following statements are true:
\begin{enumerate}
\item $B\setminus\wt B\subset Z(\wt G)$;
\item for any $u\in H(\wt G)$, $u|_{V(\wt G_x)}$ is a constant for any $x\in V(G)$;
\item for any $v\in \R^{V(G)}$ with $\vv<v,1>_{B}=0$ and $v|_{\wt B}$ being constant,
\begin{equation}
\vv<dv,dv>_{G}\geq \sigma_{|\wt B|}(\wt G)\vv<v,v>_{B}.
\end{equation}
\end{enumerate}
Here $$H(\wt G)=\{f\in \R^{V(\wt G)}\ |\ (\Delta_{\wt G}f)|_{\Omega(\wt G)}=0,\ \mbox{and}\ \vv<f,1>_{\wt B}=0.\}$$
and
$$Z(\wt G)=\{x\in \Omega(\wt G)\ |\ f(x)=0\ \mbox{for any }f\in H(\wt G).\}.$$
In particular, if $B=\wt B$ and $H(\wt G)$ separating vertices in $V(G)$, then the equalities of \eqref{eq-mono} hold for $i=1,2,\cdots,|\wt B|$ if and only if $\wt G$ is a  comb over $G$.
\end{thm}
\begin{proof}
If equalities of \eqref{eq-mono} hold for $i=1,2,\cdots, |\wt B|$, we first claim: \emph{There are $|\wt B|$ functions $u_1=1,u_2,\cdots,u_{|\wt B|}\in \R^{V(\wt G)}$ such that
\begin{enumerate}
\item[(i)] $u_i$ is an eigenfunction for $\sigma_i(\wt G,\wt B)$ for $i=1,2,\cdots, |\wt B|$;
\item[(ii)] $v_i:=u_i|_{V(G)}$ is an eigenfunction for $\sigma_i(G,B)$ for $i=1,2,\cdots, |\wt B|$;
\item[(iii)] $u_i|_{B\setminus \wt B}=0$ for $i=2,3,\cdots,|\wt B|$;
\item[(iv)]  $u_i(x)=u_i(y)$ for any edge $\{x,y\}\in E(\wt G)\setminus E(G)$ and  $i=1,2,\cdots, |\wt B|$;
\item[(v)] $\vv<u_i,u_j>_{B}=0$ for $1\leq j<i\leq |\wt B|$.
\end{enumerate}}
We will show the claim by induction. For $i\geq 2$, suppose that $u_1=1,u_2,\cdots, u_{i-1}$ has been constructed. Let  $$u_i=c_1\varphi_1+c_2\varphi_2+\cdots+c_i\varphi_i$$
with $c_1,c_2,\cdots,c_i$ are constants not all zero such that
\begin{equation}
\vv<u_i,u_j>_{B}=0
\end{equation}
for $j=1,2,\cdots,i-1$. Here $\varphi_1,\varphi_2,\cdots,\varphi_{|\wt B|}$ are the same as in the proof of Theorem \ref{thm-mono}. For the same reason as before, the existence of such constants is clear. Then,
\begin{equation}
\begin{split}
\sigma_i(G,B)=\sigma_i(\wt G,\wt B)\geq\frac{\vv<du_i,du_i>_{\wt G}}{\vv<u_i,u_i>_{\wt B}}
\geq\frac{\vv<du_i,du_i>_{G}}{\vv<u_i,u_i>_{B}}
\geq\sigma_i(G,B).
\end{split}
\end{equation}
So, the above inequalities are all equalities which implies that $u_i$ satisfies the above properties (i)--(v). This completes the proof of the claim.

Note that
$$H(\wt G)=\Span\{u_2,u_3,\cdots, u_{|\wt B|}\}.$$
So, we get (1) and (2) by (iii) and (iv). Conversely, when (1) and (2) are true, for any eigenfunction $u\in \R^{V(\wt G)}$ of $\sigma_i(\wt G)$ with $i\geq 2$, $v=u|_{V(G)}$ is an eigenfunction of $\sigma_i(\wt G)$ since
\begin{equation}
\Delta_{G}v(x)=\Delta_{\wt G}u(x)
\end{equation}
for any $x\in V(G)$. Thus, under the assumption of (1) and (2), the equalities of \eqref{eq-mono} hold for $i=1,2,\cdots, |\wt B|$ if and only if for any $v\in \R^{V(G)}$ with
\begin{equation}
\vv<v,u_i>_{B}=0
\end{equation}
for $i=1,2,\cdots, |\wt B|$,
\begin{equation}
\vv<dv,dv>_{G}\geq \sigma_{|\wt B|}(\wt G)\vv<v,v>_{B}.
\end{equation}
By (iii), we know that
$$\vv<v,u_i>_{\wt B}=\vv<v,u_i>_{B}=0$$
for $i=2,\cdots,|\wt B|$. So $v|_{\wt B}$ is constant and $\vv<v,1>_B=0$. Conversely, for any $v\in \R^{V(G)}$ such that $v|_{\wt B}$ is constant and $\vv<v,1>_B=0$, it is clear that  $\vv<v,u_i>_{B}=0$ for  $i=1,2,\cdots, |\wt B|$. This completes the proof of the first conclusion the theorem.

When $H(\wt G)$ separating vertices in $V(G)$, by (2), we know that
$$V(\wt G_x)\cap V(\wt G_y)=\emptyset$$
for any $x\neq y\in V(G)$ if the equalities of \eqref{eq-mono} hold for $i=1,2,\cdots, |\wt B|$. So, $\wt G$ is a comb over $G$. Conversely, it is clear that if $\wt G$ is a comb over $G$ and $B=\wt B$, then the equalities of \eqref{eq-mono} hold for $i=1,2,\cdots,|\wt B|$. This completes the proof of the theorem.
\end{proof}
When considering the rigidity of the case $n=im+1$ in Theorem \ref{thm-main}, we need the following result.
\begin{lem}\label{lem-graph-tree}
Let $\wt G$ be a nontrivial connected combinatorial graph and $G$ be its spanning tree. Suppose that the first positive Steklov eigenvalue of $G$ is of  multiplicity $i-1$ for some $2\leq i\leq |B(G)|$. Suppose that $\sigma_i(\wt G)=\sigma_i(G)$ and
\begin{enumerate}
\item $H_1(G)$ separates vertices in $\Omega(G)$, and
\item for any $x\in B(G)$, there is a function $f\in H_1(G)$ such that $f(x)\neq 0$,
\end{enumerate}
where $H_1(G)$ is the eigenspace of the first positive Steklov eigenvalue of $G$.
Then, $\wt G=G$.
\end{lem}
\begin{proof} Because $\sigma_i(\wt G)=\sigma_i(G)=\sigma_2(G)$, by Theorem \ref{thm-mono}, we know that
\begin{equation}
\sigma_j(\wt G)=\sigma_j(G)=\sigma_2(G)
\end{equation}
for $j=2,3,\cdots,i$. Then, by the same argument as in the proof of Theorem \ref{thm-rigidity}, there are $u_1\equiv 1,u_2,\cdots,u_i\in \R^{V(G)}$ such that
\begin{enumerate}
\item[(i)] $u_j$ is an eigenfunction both for $G$ and $\wt G$ with respect to $\sigma_j(G)$ for $j=1,2,\cdots, i$;
\item[(ii)] $u_j|_{B(G)\setminus B(\wt G)}=0$ for $j=2,3,\cdots,i$;
\item[(iii)]  $u_j(x)=u_j(y)$ for any edge $\{x,y\}\in E(\wt G)\setminus E(G)$ and  $j=2,3,\cdots, i$;
\item[(iv)] $\vv<u_j,u_k>_{B(G)}=0$ for $1\leq j<k\leq i$.
\end{enumerate}
Note that $H_1(G)$ is generated by $u_2,u_3,\cdots,u_i$. By assumption (2) and (ii), we know that $B(G)=B(\wt G)$.  Moreover, by assumption (1) and (iii), for any $x\neq y\in \Omega(G)$, $V(\wt G_x)\cap V(\wt G_y)=\emptyset$. This implies that $\wt G=G$ since $G$ is a spanning tree of $\wt G$.
\end{proof}
\section{First Steklov eigenfunction with Dirichlet boundary data and Steklov nodal domains}\label{sec-nodal}
In this section, we extend the Friedman's theory of nodal domains for Laplacian eigenfunctions to Steklov eigenfunctions. We first show that the first Steklov eigenfunction with vanishing Dirichlet boundary data is positive and the first eigenvalue is of multiplicity one which is a discrete version of Courant's result for Steklov eigenvalues.
\begin{thm}\label{thm-first-eigenfunction}
Let $(G,B,B_D,m,w)$ be a weighted connected finite graph with boundary $B$ and Dirichlet boundary $B_D$.  Suppose that the induced subgraph $G[\Omega_D]$ of $G$ on $\Omega_D$ is connected. Then, the eigenfunctions of $\lambda_1(G,B,B_D)$ must be everywhere positive or everywhere negative in $\Omega_D$.
\end{thm}
\begin{proof}
Let $f_0\in \R^{V(G)}$ be an eigenfunction for $\lambda_1(G,B_N,B_D)$. Then, $f_0$ is a minimizer of the Rayleigh quotient
\begin{equation}
R_0[f]=\frac{\vv<df,df>_G}{\vv<f,f>_{B}}
\end{equation}
among all $f\in\R^{V(G)}$ with $f|_{B}\not\equiv 0$ and $f|_{B_D}\equiv 0$. First, note that
\begin{equation}
\begin{split}
\vv<d|f|,d|f|>_G=&\sum_{\{x,y\}\in E(G)}(|f|(x)-|f|(y))^2w_{xy}\\
\leq& \sum_{\{x,y\}\in E(G)}(f(x)-f(y))^2w_{xy}\\
=&\vv<df,df>_G
\end{split}
\end{equation}
with strictly inequality when for some edge $\{x,y\}$, $f(x)<0<f(y)$. Moreover,
\begin{equation}
\vv<f_0,f_0>_{B}=\vv<|f_0|,|f_0|>_{B}.
\end{equation}
So, $|f_0|$ is also a first eigenfunction and
\begin{equation}
\vv<df_0,df_0>_G=\vv<d|f_0|,d|f_0|>_G.
\end{equation}
We then only need to show that $|f_0|>0$ on $\Omega_D$. Otherwise, because $f_0\not\equiv 0$ and $G[\Omega_D]$ is connected, there is a vertex $v\in \Omega_D$, so that $|f_0|(v)=0$ and $|f_0|(x)>0$ for some $x\in\Omega_D$ with $x\sim v$. If $v\in \Omega$, then
\begin{equation}
0=\Delta_G |f_0|(v)=\frac{1}{m_v}\sum_{y\sim v}(|f_0|(y)-|f_0|(v))w_{vy}=\frac{1}{m_v}\sum_{y\sim v}|f_0|(y)w_{vy}>0
\end{equation}
which is a contradiction. If $v\in B$, then
\begin{equation}
0=\lambda_1|f_0|(v)=\frac{\p |f_0|}{\p n}(v)=-\frac{1}{m_v}\sum_{y\sim v}|f_0|(y)w_{vy}<0
\end{equation}
which is also a contradiction. This completes the proof of the theorem.
\end{proof}
By Theorem \ref{thm-first-eigenfunction}, we have the following  straightforward consequence.
\begin{cor}\label{cor-first-eigenfunction}
Let $(G,B,B_D,m,w)$ be a weighted connected finite graph with  boundary $B$ and Dirichlet boundary $B_D$.  Suppose that the induced subgraph $G[\Omega_D]$ is connected. Then, $\lambda_1(G,B,B_D)$ is of multiplicity one and any higher eigenfunctions must change signs on $B$.
\end{cor}
\begin{proof}
If $\lambda_1$ is not of multiplicity one, then by Theorem \ref{thm-first-eigenfunction}, there are two eigenfunctions $f_1|_{B}>0,f_2|_B>0$ of $\lambda_1$ such that $\vv<f_1,f_2>_{B}=0$. This is ridiculous. So, $\lambda_1$ is of multiplicity one. Let $f$ be an eigenfunction for $\lambda_1$ with $f|_B>0$, and $g$ be an eigenfunction for $\lambda_i$ with $i\geq 2$, then $\vv<f,g>_B=0$ implies that $g$ must change signs on $B$. This completes the proof of the corollary.
\end{proof}
The following result shows that the assumption on the connectivity of $G[\Omega_D]$ is necessary for the conclusion of Theorem \ref{thm-first-eigenfunction}.
\begin{prop}\label{prop-lambda-disconnected}
Let $(G,B,B_D,m,w)$ be a weighted connected finite graph with  boundary $B$ and Dirichlet boundary $B_D$. Let $G_1,G_2,\cdots, G_k$ be the connected components of $G[\Omega_D]$, $\wt G_i=G[V(G_i)\cup B_D]$  and $B_i=B\cap V(G_i)$ for $i=1,2,\cdots,k$. Then,
$$\Spec(G,B,B_D)=\sqcup_{i=1}^k\Spec(\wt G_i,B_i,B_D).$$
Here $\Spec(G,B,B_D)$ is the collection of all the Steklov eigenvalues with vanishing Dirichlet boundary data on $(G,B,B_D,m,w)$ counting multiplicities and $\sqcup$ means disjoint union.
\end{prop}
\begin{proof}
Note that any eigenfunction of $(\wt G_i,B_i, B_D)$ is an eigenfunction of $(G,B,B_D)$ with the same eigenvalue after zero extension. This proves the conclusion.
\end{proof}
Next, we give the definition of nodal domains for Steklov eigenfunctions.
\begin{defn}
Let $(G,B,m,w)$ be a weighted connected finite graph with boundary and $K(G)$ be its geometric representation. Let  $f\in \R^{V(G)}$ be a Steklov eigenfunction of $(G,B,m,w)$ and $U$ be a connected component of $|K(G)|\setminus \wt f^{-1}(0)$. Then, $U$ and its induced graph $G_U$ are both called a nodal domain of $f$.
\end{defn}
We first confirm some elementary properties of nodal domains.
\begin{prop}\label{prop-Dirichlet-interior-connected}
Let $(G,B,m,w)$ be a weighted connected finite graph with boundary, $f\in \R^{V(G)}$ be a Steklov eigenfunction of $(G,B,m,w)$ and $U$ be a nodal domain of $f$. Then $\Omega_D(G_U)\neq\emptyset$, $G_U[\Omega_D(G_U)]$ is connected, and $B(G_U)\neq\emptyset$.
\end{prop}
\begin{proof}
By definition of $U$, there is a vertex $v\in U$ such that $f(v)\neq 0$. Then $v\in \Omega_D(G_U)$. So $\Omega_D(G_U)\neq\emptyset$. Moreover, because $U$ is path connected,  $G_U[\Omega_D(G_U)]$ is a connected graph.

Finally, if $B(G_U)=\emptyset$, let $g=\wt f|_{V(G_U)}$. Then, for any $x\in V(G_U)\setminus B_D(G_U)$, $x\in \Omega(G)$. Thus $\Delta_G f(x)=0$. Furthermore, because $f(x)\neq0$, for any $y\in V(G)$ adjacent to $x$, either $\{x,y\}\in E(G_U)$ or there is a point $z_y$ in the edge $[xy]$ such that $\wt f(z_y)=0$ and  $\{x,z_y\}\in E(G_U)$ with weight $\frac{1}{|[xz_y]|}$. In the later case, we have
\begin{equation}\label{eq-key}
\frac{g(z_{y})-g(x)}{|[xz_y]|}=\frac{\wt f(z_y)-\wt f(x)}{|[xz_y]|}=\frac{\wt f(y)-\wt f(x)}{|[xy]|}=(f(y)-f(x))w_{xy}.
\end{equation}
Therefore,
\begin{equation}
\begin{split}
\Delta_{G_U}g(x)=\Delta_Gf(x)=0
\end{split}
\end{equation}
for any $x\in V(G_U)\setminus B_D(G_U)$. So $g$ as the solution of the following boundary value problem:
\begin{equation}
\left\{\begin{array}{ll}\Delta_{G_U}g(x)=0& x\in V(G_U)\setminus B_D(G_U)\\
g(x)=0& x\in B_D(G_U),
\end{array}\right.
\end{equation}
must be vanished. This is a contradiction. So $B(G_U)\neq\emptyset$. This completes the proof of the proposition.
\end{proof}
At the end of this section, we extend Friedman's nodal domain theorem in \cite{FR93} for Laplacian eigenfunctions to the case of Steklov eigenfunctions. They are both discrete analogue of Courant's classical result.
\begin{thm}\label{thm-nodal-domain}
Let $(G, B, m,w)$ be a weighted connected finite graph with boundary and $f\in \R^{V(G)}$ be a  Steklov eigenfunction of $G$ with eigenvalue $\sigma>0$. Then, for each nodal domain $U$ of $f$, $\wt f|_{V(G_U)}$ is a Steklov eigenfunction for $G_U$ with vanishing data on $B_D(G_U)$ for the eigenvalues $\lambda_1(G_U)=\sigma$.
\end{thm}
\begin{proof} By the same argument as in the proof of Proposition \ref{prop-Dirichlet-interior-connected} using \eqref{eq-key}, we know that  $\wt f|_{V(G_U)}$ is a Steklov eigenfunction for $G_U$ with vanishing Dirichlet boundary data on $B_D(G_U)$ with respect to the eigenvalues $\sigma$. Moreover $\wt f|_{\Omega_D(G_U)}$ does not change signs because $U$ is a nodal domain.  Then, by Corollary \ref{cor-first-eigenfunction} and Proposition \ref{prop-Dirichlet-interior-connected}, we know that $\wt f|_{V(G_U)}$ is an eigenfunction for $\lambda_1(G_U)$.
\end{proof}
\section{A lower bound for the first Steklov eigenvalue with vanishing Dirichlet boundary data}\label{sec-min-lambda}
In this section, we first obtain a crucial lower bound for the first Steklov eigenvalues with vanishing Dirichlet boundary data on trees extending Lemma 3.1 in \cite{FR96} to the case of Steklov eigenvalues.  Then, as an application of the lower bound, we give a more explicit sufficient condition for the rigidity of the isodiametric estimates by He-Hua \cite{He-Hua1} and its extension by the authors \cite{YY}.

\begin{thm}\label{thm-lambda-1}
Let $(G,B,B_D,m,w)$ be a weighted finite tree with boundary $B\neq\emptyset$ and Dirichlet boundary $B_D\neq\emptyset$ such that $L(G)=B\cup B_D$, $G[\Omega_D]$ is a tree with unit weight and $|\Omega_D|=n+1$ with $n\geq 1$. Then,
\begin{equation}
\lambda_1(G,B,B_D)\geq \Lambda(l,n).
\end{equation}
The equality of the inequality holds if and only if $G[\Omega]$ is a path of length $I(l,n)$ and all the Dirichlet boundary vertices are adjacent to the same end vertex of the path and the other boundary vertices are adjacent to the other end vertex of the path. Here
\begin{equation*}
I(l,n)=\left\{\begin{array}{ll}0&l\geq n\\
\lfloor\frac{n-l}{2}\rfloor&l<n,\mbox{ and }\{\frac{n-l}{2}\}<\frac12\\
\lceil\frac{n-l}{2}\rceil&l<n,\mbox{ and }\{\frac{n-l}{2}\}>\frac12\\
\lfloor\frac{n-l}{2}\rfloor\mbox{ or }
\lceil\frac{n-l}{2}\rceil&l<n,\mbox{ and }\{\frac{n-l}{2}\}=\frac12.
\end{array}\right.
\end{equation*}
In particular, when $|B_D|=1$, the equality holds if and only if $G$ is a minimal broom $\Br(l,n)$. Here
$$l=\frac{1}{\sum_{x\in B_D,y\in \Omega_D}w_{xy}}.$$
\end{thm}
\begin{proof}
We first consider the case  $|B_D|=1$. Let $f\in \R^{V(G)}$ be an eigenfunction for $\lambda_1$. Then, by Theorem \ref{thm-first-eigenfunction}, we can assume that $f>0$ on $\Omega_D$.

 Let $B_D=\{o\}$ and  take $o$ as the root of the tree $G$. we first claim that \emph{ $f$ is increasing along the tree. More precisely, if $x$ is the parent of $y$, then $f(x)\leq f(y)$.} Indeed, if this is not true, i.e.  $f(x)>f(y)$, let
\begin{equation}
g(z)=\left\{\begin{array}{ll}f(z)+2(f(x)-f(y))&\mbox{$z=y$ or $z$ is a descendant of $y$ }\\f(z)&\mbox{otherwise}
\end{array}\right.
\end{equation}
Then,
\begin{equation}
\vv<dg,dg>_G=\vv<df,df>_G
\end{equation}
and
\begin{equation}
\vv<g,g>_{B}>\vv<f,f>_{B}
\end{equation}
and $g(o)=0$. So,
$$\frac{\vv<dg,dg>_G}{\vv<g,g>_B}<\frac{\vv<df,df>_G}{\vv<f,f>_B}$$
which contradicts that $f$ is the eigenfunction for $\lambda_1$.

Next, we claim that \emph{$f$ is strictly increasing.} In fact, when $x$ is the parent of $y\in B$,
\begin{equation}
\begin{split}
0<\lambda_1f(y)=\frac{\p f}{\p n}(y)=f(y)-f(x).
\end{split}
\end{equation}
So, $f(x)<f(y)$. Therefore, if the claim is not true, then there are three vertices $x,y,z$ such that $x$ is the parent of $y$ and $y$ is the  parent of $z$ with $f(x)=f(y)$ and $f(y)<f(z)$.  Note that $y\in\Omega$ since $y$ is not a leaf. Then
\begin{equation*}
\begin{split}
0=&\Delta_Gf(y)\\
=&(f(x)-f(y))+(f(z)-f(y))+\sum_{v\neq x,z}(f(v)-f(y))w_{vy}\\
\geq& f(z)-f(y)>0
\end{split}
\end{equation*}
since $f(v)\geq f(y)$ by that $f$ is increasing. This is a contradiction.

Let $G_0$ be a tree satisfying the assumption in the statement of the theorem achieving the minimum of $\lambda_1$ when $|\Omega_D|$  and the weight of the Dirichlet boundary edge are fixed ($G_0$ must exist because the number of such kind of trees is finite.). Let $f_0$ be the eigenfunction for $\lambda_1$ with $f_0>0$ on $\Omega_D(G_0)$. By the claims before, we know that $f_0$ is strictly increasing when taking the Dirichlet boundary vertex $o$ as the root.

We first claim that \emph{$G_0[\Omega(G_0)]$ must be a path and $o$ is adjacent to one of the end vertices of the path.} Otherwise, there are $x,y,z\in \Omega(G_0)$ such that $x$ is the parent of $y$ and $z$. Suppose that  $f_0(y)\geq f_0(z)$. Let $G'$ be the tree obtained by removing the edge $\{x,y\}$ and adding a new edge $\{y,z\}$ on $G$. Then,
\begin{equation*}
\frac{\vv<df_0,df_0>_{G'}}{\vv<f_0,f_0>_{B(G')}}=\frac{\vv<df_0,df_0>_{G_0}-(f_0(y)-f_0(x))^2+(f_0(y)-f_0(z))^2}{\vv<f_0,f_0>_{B(G_0)}}<\frac{\vv<df_0,df_0>_{G_0}}{\vv<f_0,f_0>_{B(G_0)}}.
\end{equation*}
since $0<f_0(x)<f_0(z)\leq f_0(y)$. This contradicts that $G_0$ achieves the minimum of $\lambda_1$.

We next claim that \emph{all the leaves in $B(G_0)$ must be adjacent to the other end vertex of the path $G_0[\Omega(G_0)]$ (i.e. the end vertex of $G_0[\Omega(G_0)]$ not adjacent to $o$ when $G_0[\Omega(G_0)]$ is a nontrivial path).} Otherwise, let $v$ be the other end vertex of $G_0[\Omega(G_0)]$, and suppose there is a leave $x\in B(G_0)$ adjacent to some vertex $y\in \Omega(G_0)\setminus \{v\}$. By that $f_0$ is strictly increasing, we know that
\begin{equation}
f_0(x)>f_0(y)\ \mbox{and}\ f_0(v)>f_0(y).
\end{equation}
Let $G'$ be the graph formed by removing the edge $\{x,y\}$ from $G_0$ and adding the edge $\{x,v\}$ and $f$ be such that
$$f(x)=f_0(x)-f_0(y)+f_0(v)>f_0(x)$$
and $f(z)=f_0(z)$ for $z\neq x$. Then,
\begin{equation}
\frac{\vv<df,df>_{G'}}{\vv<f,f>_{B(G')}}<\frac{\vv<df_0,df_0>_{G}}{\vv<f_0,f_0>_{B(G_0)}}=\lambda_1(G_0).
\end{equation}
This contradicts that $G_0$ achieves the minimum of $\lambda_1$.

Thus, $G_0$ must be a minimal broom $\Br(l,n)$. This completes the proof of the theorem when $|B_D|=1$.

When $|B_D|\geq 2$, let $G_0$ be the tree achieving the minimal $\lambda_1$ among all trees satisfying the assumptions in the statement of the theorem and with $|\Omega_D|$, $|B_D|$ and the weights of Dirichlet boundary edges are fixed. Let $f_0$ be a first Dirichlet eigenfunction such that $f_0|_{\Omega_D(G_0)}>0$. Let $v\in \Omega(G_0)$ with $$f(v)=\min_{x\in\Omega(G_0)} f(x).$$ Let $G'$ be the tree by moving all the Dirichlet boundary edges of $G_0$ to $v$. Then,
\begin{equation}
\frac{\vv<df_0,df_0>_{G'}}{\vv<f_0,f_0>_{B(G')}}\leq \frac{\vv<df_0,df_0>_{G}}{\vv<f_0,f_0>_{B(G_0)}}=\lambda_1(G_0).
\end{equation}
So, $G'$ is also a tree achieving the minimum of $\lambda_1$. Let $G''$ be the tree by removing all the the Dirichlet boundary edges of $G'$ and adding a Dirichlet boundary edge to $v$ with the weight equal to the total weight of all the removed Dirichlet boundary edges. Then, it is clear that $\lambda_1(G'')=\lambda_1(G')$. Now, we have reduced to the case that $|B_D|=1$ and complete the proof of the theorem.
\end{proof}

As an application of Theorem \ref{thm-lambda-1}, we can give a more explicit sufficient condition for the rigidity of the isodiametric estimate in \cite{He-Hua1} and its extension in \cite{YY}. First recall the estimate:
\begin{thm}[Theorem 5.1 in \cite{He-Hua2} \& Theorem 4.1 in \cite{YY}]\label{thm-reg-star}
For $r\geq 2$ and $l\geq 1$, let $\wt G$ be a finite tree containing $\St(r;l)$ as a subtree. Then,
\begin{equation}\label{eq-sigma-i-regular-star}
\sigma_i(\wt G)\leq\frac1l
\end{equation}
for $i=2,3,\cdots, r$. Moreover, the equality holds for $i=2$ if and only if $\wt G=\St(r;l)\vee_o\wt G_o$ with $\lambda_1\left(\wt G_o, L(\wt G_o)\setminus \{o\}, \{o\}\right)\geq \frac{1}{l}$. Here $o$ is the center of the star $\St(r;l)$, and $\St(r;l)\vee_o\wt G_o$ means the wedge-sum of $\St(r;l)$ and $\wt G_o$ at $o$.
\end{thm}
 Now, by Theorem \ref{thm-lambda-1} and Proposition \ref{prop-lambda-disconnected}. we have the following conclusion.
\begin{cor}
Let $r\geq 2$ and $l\geq 1$ and let $\wt G=\St(r;l)\vee_o\wt G_o$ be such that each branch of $\wt G_o$ with respect to $o$ has no more than $\left\lfloor \sqrt{4(l-1)+1} \right\rfloor$ edges when taking $o$ as the root of $\wt G_o$. Then, $\sigma_2(\wt G)=\frac{1}{l}$.
\end{cor}
\begin{proof}
Let $\wt G_1,\cdots,\wt G_k$ be the branches of $\wt G_o$. Then, by Theorem \ref{thm-lambda-1},
\begin{equation}
\lambda_1(\wt G_i,L(\wt G_i)\setminus\{o\},\{o\})\geq\Lambda(|E(\wt G_i)|)=\frac{1}{1+\left\lfloor\frac{|E(\wt G_i)|^2}{4}\right\rfloor}\geq\frac1l.
\end{equation}
So, by Proposition \ref{prop-lambda-disconnected},
\begin{equation}
\lambda_1(\wt G_o,L(\wt G_o)\setminus\{o\},\{o\})\geq \frac1l.
\end{equation}
This completes the proof of the theorem.
\end{proof}
\section{Clump numbers of trees}\label{sec-clump}
In this section, we recall the definition of the clump number of a tree in \cite{FR96} and introduce some of its properties.
\begin{defn}
Let $G$ be a finite combinatorial tree and  $K(G)$ be its geometric representation. For any  $p\in |K(G)|$, the connected components $U_1, U_2,\cdots, U_k$ (or their corresponding induced trees $G_{U_1},G_{U_2}, \cdots, G_{U_k}$) of $|K(G)|\setminus\{p\}$ are called the clumps of $G$ with respect to $p$. The clump number of $G$ with respect to $p$ is defined as
\begin{equation*}
\Clump(G,p)=\max\{\Length(U_1),\Length(U_2),\cdots, \Length(U_k)\}
\end{equation*}
where $\Length (U_i)$ is the total length of $U_i$. Moreover, the clump number of $G$ is defined to be
\begin{equation}
\Clump(G)=\inf_{p\in |K(G)|}\Clump(G,p).
\end{equation}
\end{defn}
We have the following properties for $\Clump(G,p)$.
\begin{prop}\label{prop-clump}
Let $G$ be a finite combinatorial tree and $K(G)$ be its geometric representation. Then,
\begin{enumerate}
\item $\Clump(G,p)$ is a lower semi-continuous function for $p\in |K(G)|$. Thus there is a point $p\in |K(G)|$ such that $\Clump(G)=\Clump(G,p)$;
\item the point $p\in |K(G)|$ such that $\Clump(G)=\Clump(G,p)$ is either a vertex or a mid-point of an edge. Thus, $\Clump(G)$ is either an integer or a half-integer.
\item if $p\in |K(G)|$ is a mid-point of some edge such that $\Clump(G)=\Clump(G,p)$, then the clumps of $G$ with respect to $p$ are two clumps of equal total length $\frac{|E(G)|}{2}$;
\item there is unique point $p\in |K(G)|$ such that $\Clump(G)=\Clump(G,p)$. We called the  point $p$ the equilibrium point of $G$.
\end{enumerate}
\end{prop}
\begin{proof}
(1) When $p$ is not a vertex of $K(G)$, it is clear that $\Clump(G,p)$ is continuous at $p$ by definition. When $p$ is a vertex of $K(G)$, let $U_1,U_2,\cdots, U_k$ be all the clumps of $p$ and suppose that $U_1$ has the maximal total length among the $k$ clumps. Then, for any $\e\in (0,1)$ and $x\in B_p(\e)\cap U_i$ with $i=2,3,\cdots, k$,
\begin{equation}
\Clump(G,x)\geq \Clump(G,p)\geq \Clump(G,p)-\e
\end{equation}
because there is a clump of $x$ containing $U_1$. Moreover, for any $x\in B_p(\e)\cap U_1$, it is clear that
\begin{equation}
\Clump(G,x)\geq \Clump(G,p)-\e
\end{equation}
since $U_1\setminus [px]$ is a clump of $x$. So, $\Clump(G,p)$ is a lower semi-continuous function of $p\in |K(G)|$.

(2) Let $p\in |K(G)|$ be a point such that $\Clump(G)=\Clump(G,p)$ and $p$ is not a vertex of $K(G)$. Suppose that $p$ is contained in the edge $\ol{xy}$. Let $U_x$ and $U_y$ be the clumps of $p$ containing $x$ and  $y$ respectively. Then
\begin{equation}
\Length(U_x)=|[xp]|+n_x\ \mbox{and}\ \Length(U_y)=|[yp]|+n_y
\end{equation}
where $n_x+1$ and $n_y+1$ are numbers of vertices in $U_x$ and $U_y$ respectively.
We claim that \emph{$n_x=n_y$}. Otherwise, suppose that $n_x>n_y$. Let $q$ be the mid-point of $[xp]$. Then,
\begin{equation*}
\Clump(G,q)=n_x+|[xq]|=n_x+\frac12{|[xp]|}<n_x+|[xp]|=\Clump(G,p)=\Clump(G)
\end{equation*}
which is a contradiction. Now, let $z$ be the mid-point of the edge $\ol{xy}$. Then,
\begin{equation*}
\Clump(G)\leq\Clump(G,z)=n_x+\frac12\leq\Clump(G,p)=\Clump(G).
\end{equation*}
So $p=z$ is the the mid-point of the edge $\ol{xy}$.

(3) It has been shown in the proof of (2).

(4) Let $p,q\in |K(G)|$ be two different points such that $$\Clump(G,p)=\Clump(G,q)=\Clump(G).$$
By (2) and (3), we know that $p$ and $q$ must be simultaneously vertices of $K(G)$ or mid-points of some edges of $K(G)$.

For the first case: $p$ and $q$ are both vertices of $K(G)$, let $U_1$ be the clump of $p$ that contains $q$ and $W_1$ be the clump of $q$ that contains $p$. Then,
\begin{equation}
\Clump(G)=\Clump(G,p)\geq \Length(U_1)
\end{equation}
and similarly,
\begin{equation}
\Clump(G)=\Clump(G,q)\geq \Length(W_1).
\end{equation}
Let $v$ be the vertex in $U_1$ adjacent to $p$ and $z$ be the mid-point of $\ol{pv}$.  Let $U_p$ and $U_q$ be the clumps of $z$ containing $p$ and $q$ respectively. Then, $U_p\subsetneqq W_1$ and $U_q\subsetneqq U_1$. So,
\begin{equation}
\Length(U_p)<\Length(W_1)\ \mbox{and}\ \Length(U_q)<\Length (U_1)
\end{equation}
which implies that
\begin{equation}
\Clump(G,z)=\max\{\Length(U_p),\Length(U_q)\}<\Clump(G).
\end{equation}
This is a contradiction.

For the second case: both $p$ and $q$ are mid-points of some edges, it is clear that  this is impossible by (3).
\end{proof}
For further applications, we need the following notions and propositions come from \cite{FR96}. To make the paper more self-contained, we quote them at the end of this section. All the trees are assumed to be finite combinatorial trees.
\begin{prop}[Lemma 4.3 in \cite{FR96}]\label{prop-clump-tree}
Let $G$ be a finite tree. Then, $$\Clump(G)\leq \frac{|E(G)|}{2}.$$
\end{prop}
\begin{prop}[Lemma 5.1 in \cite{FR96}]\label{prop-clump-tree-1}
Let $G$ be a finite tree. Suppose that $|E(G)|\leq (r+2)k+r$ for some integers $r\geq 0$ and $k\geq 1$. Then, we can remove at most $r$ edges from $G$ to get a forest each of whose trees have clump number not greater than $k$.
\end{prop}
\begin{prop}[Lemma 5.2 in \cite{FR96}]\label{prop-clump-tree-2}
Let $G$ be a finite tree with $|E(G)|\leq (r+2)k+(r+1)$ edges for some integers $k\geq 0$ and $r\geq 0$. Then, we can remove at most $r$ edges from $G$ to get a forest each of whose trees have clump number not greater than $k+\frac12$.
\end{prop}

\begin{defn}[Definition 8.1 \& 8.2 in \cite{FR96}]
Let $k\geq 1$ be an integer.

 (1) A tree $G$ is called of type A with respect to $k$ if $|E(G)|=rk-1$ for some $r\geq 1$ and  by removing $r-1$ edges we are left with a forest of $r$ trees each of $k-1$ edges.

 (2) A tree $G$  is called of type B if  $(r-1)k\leq|E(G)|\leq rk-1$ for some $r\geq 2$ and by removing $r-2$ edges we are left with a forest of trees whose clump number are not greater than $k-1$.
\end{defn}
\begin{prop}[Lemma 8.3 in \cite{FR96}]\label{prop-tree-AB}
Let $G$ be finite tree and $k\geq 1$ be an integer. If $|E(G)|\geq k-1$  then $G$ is  either of type A or type B with respect to $k$.
\end{prop}
\begin{rem}
The statement of the Proposition \ref{prop-tree-AB} is slightly more general than the original statement of Lemma 8.3 in \cite{FR96}. The proof is the same as Friedman's original proof. We state it in this more general form because in Friedman's original argument by induction, it seems that such a general form of statement is needed.
\end{rem}
\section{Minimal Steklov eigenvalues}\label{sec-min-sigma}
In this section, we prove Theorem \ref{thm-main}. The following conclusion is crucial in our solution of the extremal problems for Steklov eigenvalues.
\begin{thm}\label{thm-Steklov-clump}
Let $G$ be a nontrivial finite combinatorial tree. Then,
\begin{equation}\label{eq-steklov-clump}
\sigma_2(G)\geq \Lambda(\Clump(G)).
\end{equation}
The equality holds if and only if at least two of the clumps of the equilibrium point of $G$ are the minimal brooms $\Br(\Clump(G))$.
\end{thm}
\begin{proof}
Let $p\in |K(G)|$ be the equilibrium point of $G$ and  $f\in \R^{V(G)}$ be an eigenfunction for $\sigma_2(G)$. Suppose that $\wt f(p)\leq 0$. Let $U$ be a clump of $p$ containing a leaf $v$ of $G$ such that $f(v)>0$ and $N$ be the nodal domain of $f$ containing $v$. It is clearly that $N\subset U$. By Theorem \ref{thm-nodal-domain}, $g=\wt f|_{V(G_N)}$ is an eigenfunction for $\lambda_1(G_N)=\sigma_2(G)$. Let $\ol g\in \R^{V(G_{U})}$ be the zero extension of $g$. That is $g(x)=0$ for any $x\in V(G_{U})\setminus V(G_N)$. Then, by Theorem \ref{thm-nodal-domain} and Theorem \ref{thm-lambda-1},
\begin{equation}\label{eq-zero-extension}
\begin{split}
\sigma_2(G)=\lambda_1(G_N)=&\frac{\vv<dg,dg>_{G_N}}{\vv<g,g>_{B(G_N)}}\geq \frac{\vv<d\ol g,d\ol g>_{G_U}}{\vv<\ol g,\ol g>_{B(G_U)}}\\
\geq&\lambda_1(G_{ U})\geq \Lambda(\Length(U))\geq\Lambda(\Clump(G)),
\end{split}
\end{equation}
where the first inequality comes from the fact that each edge of $N$ is contained in some edge of $U$.

Moreover, if $\sigma_2(G)=\Lambda(\Clump(G))$, then $U=N$ and $\wt f(p)=0$. Furthermore, by Theorem \ref{thm-lambda-1}, $G_U$ must be a minimal broom  $\Br(\Clump(G))$. Similarly, let $W$ be a clump of $p$ containing a leaf $w$ of $G$ such that $f(w)<0$. Then, for the same reason, $G_W$ must be a minimal broom  $\Br(\Clump(G))$. This completes the necessary part for the equality to hold.

Conversely, if there are two clumps $U,W$ of $p$ that are minimal brooms $\Br(\Clump(G))$. Let $f_U$ be the eigenfunction and $f_W$ be the eigenfunctions of $G_U$ and $G_W$ respectively such that
$$\vv<f_U,1>_{B(G_U)}+\vv<f_W,1>_{B(G_W)}=0.$$
Let  $\psi\in \R^{V(G)}$ be such that
\begin{equation}
\psi(x)=\left\{\begin{array}{ll}f_U(x)&x\in V(G_U)\\
f_W(x)&x\in V(G_W)\\
0&\mbox{otherwise.}
\end{array}\right.
\end{equation}
Then, $\psi$ is a Steklov eigenfunction of $G$ with eigenvalue $\Lambda(\Clump(G))$ by Lemma \ref{lem-lambda1-broom}. By the estimate \eqref{eq-steklov-clump}, we know that
\begin{equation}
\sigma_2(G)=\Lambda(\Clump(G)).
\end{equation}
This completes the proof of the theorem.
\end{proof}
As an application of the Theorem \ref{thm-Steklov-clump}, we can solve the extremal problem for $\sigma_2(G)$ when $G$ is a tree.
\begin{thm}\label{thm-sigma2-tree}
Let $G$ be a nontrivial finite combinatorial tree. Then,
\begin{equation}\label{eq-sigma2-tree}
\sigma_2(G)\geq \Lambda\left(\frac{|E(G)|}{2}\right).
\end{equation}
The equality holds if and only if \begin{enumerate}
    \item $G=\DB(m,2m,m)$ when $|E(G)|=4m$;
    \item $G=\DB(m,2m+1,m)$ when $|E(G)|=4m+1$;
    \item $G=\DB(m+1,2m,m+1), \DB(m,2m+2,m)\ \mbox{or}\ \DB(m,2m+1,m+1)$ when $|E(G)|=4m+2$;
    \item $G=\DB(m+1,2m+1,m+1)$ when $|E(G)|=4m+3$.
    \end{enumerate}
\end{thm}
\begin{proof}
By Proposition \ref{prop-clump-tree}, we know that
$$\Clump(G)\leq \frac{|E(G)|}{2}.$$
Thus, by Theorem \ref{thm-Steklov-clump},
\begin{equation}
\sigma_2(G)\geq \Lambda(\Clump(G))\geq \Lambda\left(\frac{|E(G)|}{2}\right).
\end{equation}
When the equality of \eqref{eq-sigma2-tree} holds, we know that $\Clump(G)=\frac{|E(G)|}{2}$. Moreover, by the rigidity part of Theorem \ref{thm-Steklov-clump}, we get the conclusion.
\end{proof}
By combining Theorem \ref{thm-sigma2-tree} and Lemma \ref{lem-graph-tree}, we can prove (1) of Theorem \ref{thm-main}.

\begin{proof}[Proof for (1) of Theorem \ref{thm-main}]
Let $T$ be a spanning tree of $G$. Then, by Theorem  \ref{thm-mono},
\begin{equation}
\sigma_2(G)\geq \sigma_2(T)\geq \Lambda\left(\frac{|E(T)|}{2}\right)=\Lambda\left( \frac{n-1}{2}\right).
\end{equation}
Conversely, if $\sigma_2(G)=\Lambda\left( \frac{n-1}{2}\right)$, by the rigidity part of Theorem \ref{thm-sigma2-tree}, $T$ must be a dumbbell listed in Theorem \ref{thm-sigma2-tree}. Moreover, by Lemma \ref{lem-dumbbell} and Lemma \ref{lem-graph-tree}, $G=T$. This completes the proof.
\end{proof}
Next, we come to deal with the case $i\not|n$ in Theorem \ref{thm-main}. By imitating the arguments in \cite{FR96}, we need to modify the notion of sub-$k$ in \cite[Definition 5.3]{FR96}.
\begin{defn}
Given a positive integer $k$, a finite combinatorial tree $G$ is called sub-$k$ if either (i) it has clump number less than $k$, or (ii) it has clump number $k$ with respect to some vertex $o$, but contains at most one clump with respect to $o$ which is a minimal broom $\Br(k)$ with $o$ as the root.
\end{defn}
By Theorem \ref{thm-Steklov-clump}, if a tree $G$ is sub-$k$, then
\begin{equation}
\sigma_2(G)>\Lambda(k).
\end{equation}
Similarly as in \cite[Lemma 5.4]{FR96}, we have the following result with the proof similar to the proof of \cite[Lemma 5.4]{FR96}. For completeness, we give the proof of the result in details.
\begin{prop}\label{prop-tree-sub-k}
Let $G$ be a combinatorial tree with $|E(G)|=(r+2)k$ edges for some integers $r\geq 0$ and $k\geq 1$. Then, unless $G$ is  star of degree $r+2$ with each arm a minimal broom $\Br(k)$ with the center of the star as the root,  we can remove at most $r$ edges from $G$ to get a forest, each of whose trees are sub-$k$.
\end{prop}
\begin{proof}
We proceed by induction.

When $r=0$, if $G$ is not sub-$k$, then $\Clump(G)\geq k$ and there is  vertex $o$ such that there are least two clumps with respect to $o$ that are minimal brooms $\Br(k)$ with $o$ as the root. This means that $G$ is a star with two arms that are  minimal brooms $\Br(k)$. This proves the conclusion for $r=0$.

When $r=1$, let $p$ be the equilibrium point of $G$. By Proposition \ref{prop-clump-tree},
 $$\Clump(G,p)\leq |E(G)|/2=\frac{3k}{2}.$$
If $p$ is an mid-point of some edge, by replacing $p$ by an end-vertex $v$ of that edge,  we know that there is a vertex $v$ of $G$ such that
\begin{equation}
\Clump(G,v)\leq2k.
\end{equation}
Let $T$ be a clump of $G$ with respect to $v$ such that
$$|E(T)|=\Clump(G,v)\leq 2k.$$

If $|E(T)|>k$, after removing the edge in $T$ joining to $v$, we get two trees $T_1$ and $T_2$ such that
$$k\leq |E(T_1)|=|E(T)|-1\leq 2k-1$$
and
$$k\leq|E(T_2)|=|E(G)|-|E(T)|\leq 2k-1.$$
By Proposition \ref{prop-clump-tree}, we know that $T_1$ and $T_2$ are both sub-$k$.

If $|E(T)|\leq k$ and $G$ is not a star with center $v$ and each arm a minimal broom $\Br(k)$, then there are at most two clumps of $G$ that are minimal brooms $\Br(k)$. If there is no clump of $G$ with respect to $v$ that is a minimal broom $\Br(k)$, the $G$ itself is sub-$k$. So no edge need to be removed. If there is a clump of $G$ that is a minimal broom $\Br(k)$, remove the edge in that clump that is adjacent to $v$, we get a forest of two trees whose trees are both sub-$k$ because there are at most two clumps of $G$ that are minimal brooms $\Br(k)$. This proves the conclusion for $r=1$.

Suppose the conclusion is true for $r\leq m-1$  with $m\geq 2$. When $r=m$, let  $p$ be the equilibrium point of  $G$. Then, by Proposition \ref{prop-clump-tree},
 $$\Clump(G,p)\leq \frac{|E(G)|}2=\frac{(m+2)k}{2}.$$
 If $p$ is the mid-point of some edge, by replacing $p$ by an end-vertex $v$ of that edge,  we know that there is a vertex $v$ of $G$ such that
\begin{equation}
\Clump(G,v)\leq mk.
\end{equation}

If there is a clump $T$ of $v$ with $|E(T)|=kq$ with $ 2\leq q\leq m$, then after removing the edge in $T$ adjacent to $v$, the tree $G$ will break into two trees $T_1$ and $T_2$ with
\begin{equation}
|E(T_1)|=|E(T)|-1=kq-1=q(k-1)+q-1
\end{equation}
and
\begin{equation}
|E(T_2)|=|E(G)|-|E(T)|=(m+2-q)k.
\end{equation}
By Proposition \ref{prop-clump-tree-2}  and the induction hypothesis, we can remove no more than $q-2$ edges from $T_1$ to get a forest whose trees are of clump numbers not greater than $k-\frac12<k$, and we can also remove $m-q$ edges from $T_2$ to get a forest whose trees are all sub-$k$. Note that we have removed no more than
$$1+q-2+m-q=m-1$$
edges from $G$. We get the conclusion for this case when $r=m$.

If there is a clump $T$ of $v$ with clump number greater than $k$ but not a multiple of $k$, let $|E(T)|=kq+i$ with $i=1,2,\cdots, k-1$ and $q\geq 1$. Then, by removing the edge in $T$ adjacent to $v$, we get two trees $T_1$ and $T_2$ such that
$$|E(T_1)|=kq+i-1\leq qk+k-2<(q+1)(k-1)+q$$ and
$$|E(T_2)|\leq (m-q+2)k-1=(m-q+2)(k-1)+(m-q+1).$$ So, by Proposition \ref{prop-clump-tree-2}, we can remove $q-1$ and $m-q$ edges from $T_1$ and $T_2$ respectively to get a forest whose trees are of clump numbers not greater than $k-\frac12<k$. This gives us the conclusion in this case when $r=m$.

The remaining case is that each clump of $v$ is of total length not greater than $k$. Let $P_1,P_2,\cdots, P_t, T_1,T_2,\cdots,T_s$ be clumps of $G$ with respect to $v$. Such that $P_1,P_2,\cdots, P_t$ are minimal brooms $\Br(k)$ and $T_1,\cdots,T_s$ are either trees  of total length not greater $k-1$ or  trees of of total length $k$ that are not a minimal broom $\Br(k)$. If $t=m+2$, then $G$ is a star with degree $m+2$ and each arm a minimal broom $\Br(k)$. If $t\leq m+1$, then by removing the edges in $P_1,P_2,\cdots, P_{t-1}$ adjacent to $v$, we get  a forest whose trees are all sub-$k$. This completes the proof of the conclusion.
\end{proof}
We are now ready to solve the extremal problem for trees when $i\not||V(G)|$.
\begin{thm}\label{thm-sigma-i-tree}
Let $G$ be a finite combinatorial tree and  $2<i<|V(G)|$ be an integer with $i\not|\ |V(G)|$. Then
\begin{equation}\label{eq-sigma-i-tree}
\sigma_i(G)\geq \Lambda(m)
\end{equation}
where $m=\left\lfloor \frac{V(G)}i\right\rfloor$. Moreover,  when $|V(G)|=im+1$, the equality holds if and only if $G$ is a star of degree $i$ such that each arm is a minimal broom $\Br(m)$ with the center of the star as the root.
\end{thm}
\begin{proof}
Note that
$$|E(G)|=|V(G)|-1\leq mi+i-2.$$
So, by Proposition \ref{prop-clump-tree-1}, one can remove no more than $i-2$ edges from $G$ to get a forest $F$ whose trees $T_1,T_2,\cdots, T_{k}$ are of clump numbers not greater than $m$. Note that $k\leq i-1$. So, by Theorem \ref{thm-mono} and Theorem \ref{thm-Steklov-clump},
\begin{equation}
\sigma_i(G)\geq\sigma_i(F)\geq \min\{\sigma_2(T_1),\cdots,\sigma_2(T_k)\}\geq \Lambda(m)
\end{equation}
by noting that the Steklov spectrum of the forest $F$ is the disjoint union of the Steklov spectrums of its trees: $T_1,T_2,\cdots, T_k$.

When $|V(G)|=mi+1$, that is $|E(G)|=mi$, if $G$ is  not a star of degree $i$ with each arm a minimal broom $\Br(m)$, by Proposition \ref{prop-tree-sub-k}, one can remove no more than $i-2$ edges to get a forest whose trees: $T_1,T_2,\cdots, T_k$ are all sub-$m$.  In this case, similarly as before,
\begin{equation}
\sigma_i(G)\geq \sigma_i(F)\geq \min\{\sigma_2(T_1),\cdots,\sigma_2(T_k)\}> \Lambda(m).
\end{equation}
This means that if the equality of \eqref{eq-sigma-i-tree} is achieved, then $G$ must be a  star of degree $i$ with each arm a minimal broom $\Br(m)$. This completes the proof of theorem. \end{proof}
By Theorem \ref{thm-sigma-i-tree} and Lemma \ref{lem-graph-tree}, we can prove (2) of Theorem \ref{thm-main}
\begin{proof}[Proof for (2) of Theorem \ref{thm-main}]
Let $T$ be a spanning tree of $G$. Then, by Theorem \ref{thm-mono} and Theorem \ref{thm-sigma-i-tree},
\begin{equation}
\sigma_i(G)\geq \sigma_i(T)\geq \Lambda(m).
\end{equation}
When $V(G)=mi+1$ and the equality holds, by Theorem \ref{thm-sigma-i-tree}, $T$  must be a star of degree $i$ with each arm a minimal broom $\Br(m)$. Finally, by Lemma \ref{lem-star-broom} and Lemma \ref{lem-graph-tree}, $G=T$. This completes the proof.
\end{proof}
We finally come to deal with the case $i|n$ in Theorem \ref{thm-main}. We first solve the extremal problem for trees.
\begin{thm}
Let $G$ be a finite combinatorial tree and $2<i<|V(G)|$ with $i|\ |V(G)|$. Then
\begin{equation}
\sigma_i(G)\geq \Lambda(m-1+\theta_i)
\end{equation}
where $m=\frac{V(G)}{i}$ and $\theta_i=\frac{1}{4\cos^2\frac{\pi}{2i}}$. The equality holds if and only if $G$ is the regular comb $\Comb(P_i;T)$ where $P_i$ is a path on $i$ vertices and $T$ is the minimal broom $\Br(m-1+\theta_i)$ with the Dirichlet boundary vertex deleted and the vertex adjacent to Dirichlet boundary vertex as the root.
\end{thm}
\begin{proof}
Note that $|E(G)|=im-1$, by Proposition \ref{prop-tree-AB}, $G$ is either of type $A$ or type $B$.

When $G$ is of type $B$, by removing $i-2$ edges from $G$ we are left a forest $F$ consisting of $i-1$ trees: $T_1,T_2,\cdots, T_{i-1}$ with clump numbers all no more than $m-1$. Then, by Theorem \ref{thm-mono} and Theorem \ref{thm-Steklov-clump},
\begin{equation*}
\begin{split}
\sigma_i(G)\geq \sigma_i(F)\geq\min\{\sigma_2(T_1),\cdots,\sigma_2(T_{i-1})\}\geq \Lambda(m-1)>\Lambda(m-1+\theta_i).
\end{split}
\end{equation*}

When $G$ is of type A, we can remove $i-1$ edges to get a forest $F$ of $i$ trees $T_1,T_2,\cdots, T_i$ that are all of $m-1$ edges. Let $G'$ be a graph with
\begin{equation*}
V(G')=\{T_1,T_2,\cdots, T_i\}
\end{equation*}
and
\begin{equation*}
E(G')=\{\{T_r,T_s\}\ |\ \mbox{ There is an edge joining}\ T_r
\  {\rm and}\ T_s\ {\rm in }\ G.\}
\end{equation*}
It is clear that $G'$ is a tree. Let $\varphi$ be a Laplacian eigenfunction of $G'$ for $\mu_i(G')$. By Proposition \ref{prop-top-eigen}, $\varphi$ has alternating signs. So, on each edge $e$ of $|K(G')|$, there is zero point $z_e$ of $\wt\varphi$. Because each edge of $e$ corresponding to an edge in $G$, we can simply view $z_e$ as a point on the corresponding edge of $|K(G)|$. Let $\wt T_1,\wt T_2,\cdots, \wt T_{i}$ be the connected components of $|K(G)|$ with all the $z_e$'s deleted such that
$\wt T_j\supset T_j$ for $j=1,2,\cdots, i$. By Corollary \ref{cor-top-eigen},
\begin{equation}\label{eq-weight}
\sum_{e\in E(B_D(\wt T_j),\Omega_D(\wt T_j))}w_e=\mu_i(G').
\end{equation}
for $j=1,2,\cdots, i$.

Let $f_1\equiv 1,f_2,\cdots,f_{i}\in \R^{V(G)}$ be an orthogonal system of Steklov eigenfunctions for $G$ such that $f_j$ is an eigenfunction of $\sigma_j(G)$ for $j=1,2,\cdots,i$. Let $U=\mbox{span}\{f_1,f_2,\cdots, f_{i}\}$ and $$W=\{f\in \R^{V(G)}\ | \ \wt f(z_e)=0\mbox{ for all }z_e.\}.$$
Note that $\dim U=i$ and $\dim W\geq |V(G)|-i+1$. So $U\cap W\neq 0$. Let $g\in U\cap W $ be a nonzero function. Then, by Theorem \ref{thm-lambda-1}, Proposition \ref{prop-min-top} and \eqref{eq-weight},
\begin{equation*}
\begin{split}
\sigma_i(G)\geq\frac{\vv<dg,dg>_{G}}{\vv<g,g>_{B(G)}}=&\frac{\sum_{j=1}^i\vv<d\wt g,d\wt g>_{\wt T_j}}{\sum_{j=1}^i\vv<\wt g,\wt g>_{B(\wt T_j)}}\geq \min_{1\leq j\leq i}\left\{\frac{\vv<d\wt g,d\wt g>_{\wt T_j}}{\vv<\wt g,\wt g>_{B(\wt T_j)}}\right\}\\
\geq&\min\{\lambda_1(\wt T_1),\lambda_1(\wt T_2),\cdots,\lambda_1(\wt T_i)\}\\
\geq&\Lambda\left(\frac{1}{\mu_i(G')},m-1\right)\geq\Lambda\left(\frac{1}{\mu_i(P_i)},m-1\right)=\Lambda(m-1+\theta_i).\\
\end{split}
\end{equation*}
When the equality holds, we know that $G$ must be of type A and $G'$ must be a path on $i$ vertices. Moreover, the function $g$ before must be an eigenfunction for $\sigma_i(G)$. Let $N_1,N_2,\cdots, N_s$ be the nodal domains of $g$. It is clear that each $N_j$ is contained in some $\wt T_\nu$. Then, by Theorem \ref{thm-nodal-domain} and Theorem \ref{thm-lambda-1},
\begin{equation}
\Lambda(m-1+\theta_i)=\sigma_i(G)=\lambda_1(N_j)\geq \lambda_1(\wt T_\nu)\geq \Lambda(m-1+\theta_i).
\end{equation}
The inequality $\lambda_1(N_j)\geq \lambda_1(\tilde T_\nu)$ comes from the same argument as in \eqref{eq-zero-extension} by zero extension. Thus, by the rigidity part of Theorem \ref{thm-lambda-1} and Corollary \ref{cor-sigma-i-comb}, we completes the proof of the theorem.
\end{proof}
Finally, we come to prove (3) of Theorem \ref{thm-main}.
\begin{proof}[Proof for (3) of Theorem \ref{thm-main}]
Let $G'$ be a spanning tree of $G$. By Theorem \ref{thm-mono},
\begin{equation}
\sigma_i(G)\geq\sigma_i(G')\geq \Lambda(m-1+\theta_i)
\end{equation}
When the equality holds, by the rigidity part of the last theorem, we know that every spanning tree of $G$ must be isomorphic to $\Comb(P_i;T)$. By \cite[Theorem in P.424]{Ve}, we know that $G=\Comb(P_i;T)$ or $\Comb(C_i;T)$.  When $i$ is even, because $$\mu_i(C_i)=4>\mu_i(P_i),$$
by Corollary \ref{cor-sigma-i-comb}, we know that
\begin{equation}
\sigma_i(\Comb(C_i,T))>\sigma_i(\Comb(P_i,T)).
\end{equation}
When $i$ is odd, by \cite[P. 9]{BR},
$$\mu_i(C_i)=\mu_i(P_i).$$
So, by Corollary \ref{cor-sigma-i-comb},
\begin{equation}
\sigma_i(\Comb(C_i,T))=\sigma_i(\Comb(P_i,T))=\Lambda(m-1+\theta_i).
\end{equation}
This completes the proof.
\end{proof}

\end{document}